\documentclass[12pt, twoside]{article}
\usepackage{amsmath,amsthm,amssymb}
\usepackage{times}
\usepackage{enumerate}

\usepackage{color}

\def\titlerunning#1{\gdef\titrun{#1}}
\makeatletter
\def\author#1{\gdef\autrun{\def\and{\unskip, }#1}\gdef\@author{#1}}
\def\address#1{{\def\and{\\\hspace*{18pt}}\renewcommand{\thefootnote}{}%
\footnote {#1}}%
\markboth{\autrun}{\titrun}}
\makeatother
\def\email#1{e-mail: #1}
\def\subjclass#1{{\renewcommand{\thefootnote}{}%
\footnote{\emph{Mathematics Subject Classification (2010):} #1}}}
\def\keywords#1{\par\medskip
\noindent\textbf{Keywords.} #1}

\newtheorem{thm}{Theorem}[section]
\newtheorem{cor}[thm]{Corollary}
\newtheorem{lem}[thm]{Lemma}
\newtheorem{conj}[thm]{Conjecture}
\newtheorem{proposition}[thm]{Proposition}

\theoremstyle{definition}
\newtheorem{defin}[thm]{Definition}
\newtheorem{rem}[thm]{Remark}
\newtheorem{example}[thm]{Example}

\def\oc{\mathcal{O}}

\def\vol{\operatorname{vol}}
\def\ord{\operatorname{ord}}

\numberwithin{equation}{section}

\begin{document}

%%%%% To ease editing, add:

\baselineskip=17pt

%%%%%%%%%%%%%%%%

%% In the running head, give an abbreviation of the title. 
\titlerunning{Generalized Okounkov Bodies}

\title{\textbf{Transcendental Morse Inequality\\and Generalized Okounkov Bodies}}

\author{Ya Deng}

\date{}

\maketitle

\address{Ya Deng:Institut Fourier, Universit\'e de Joseph Fourier, 100 rue des Maths, 38402 Saint-Martin d'H\`eres, France, and School of Mathematical Sciences, University of Science and Technology of China;
\email{Ya.Deng@ujf-grenoble.fr}}

\subjclass{Algebraic geometry; Several complex variables and analytic spaces}

%%%%%%%%

\begin{abstract}
The main goal of this article is to construct ``arithmetic Okounkov bodies" for an arbitrary pseudo-effective (1,1)-class $\alpha$ on a K\"ahler manifold. Firstly, using Boucksom's divisorial Zariski decompositions for  pseudo-effective (1,1)-classes on compact K\"ahler manifolds, we prove the  differentiability of volumes of big classes for K\"ahler manifolds on which modified nef cones and nef cones coincide; this includes K\"ahler surfaces. We then apply our differentiability results to prove Demailly's transcendental Morse inequality for these particular classes of K\"ahler manifolds. In the second part, we construct the convex body $\Delta(\alpha)$ for any big class $\alpha$ with respect to a fixed flag by using positive currents, and prove that this newly defined convex body coincides with the Okounkov body when $\alpha\in {\rm NS}_{\mathbb{R}}(X)$; such convex sets $\Delta(\alpha)$ will be called generalized Okounkov bodies. As an application we prove that any rational point in the interior of Okounkov bodies is ``valuative". Next we give a complete characterisation of generalized Okounkov bodies on surfaces, and show that the generalized Okounkov bodies behave very similarly to original Okounkov bodies. By the differentiability formula, we can relate the standard Euclidean volume of $\Delta(\alpha)$ in $\mathbb{R}^2$ to the volume of a big class $\alpha$, as defined by Boucksom; this solves a problem raised by Lazarsfeld in the case of surfaces. Finally, we study the behavior of the generalized Okounkov bodies on the boundary of the big cone, which are characterized by numerical dimension.
\keywords{Generalized Okounkov body, positive current, Siu decomposition, divisorial Zariski decomposition, Lelong number, transcendental Morse inequality,
numerical dimension, differentiability formula}
\end{abstract}

%%%%%%%%%%%%%%%%%%%%%%%%%%%%%%%%%%%%%%%%%%%%%

\section{Introduction}

In \cite{Oko96} Okounkov introduced a natural procedure to associate a convex body $\Delta(D)$ in $\mathbb{R}^n$ to any ample divisor $D$ on an $n$-dimensional projective variety. Relying on the work of Okounkov, Lazarsfeld and Musta\c{t}\u{a} \cite{LM09}, and Kaveh and Khovanskii \cite{KK09,KK10}, have systematically studied Okounkov's construction, and associated to any big divisor and any fixed flag of subvarieties a convex body which is now called the Okounkov body. 

We now briefly recall the construction of the Okounkov body.  We start with a complex projective variety $X$ of dimension $n$. Fix a flag
$$ Y_{\bullet} : X=Y_0\supset Y_1 \supset Y_2 \supset \ldots \supset Y_{n-1} \supset Y_{n}=\{p\}
$$
where $Y_i$ is a smooth irreducible subvariety of codimension $i$ in $X$. For a given big divisor $D$, one defines a valuation-like function
$$
\mu=\mu_{Y_{\bullet},D}: (H^0(X,\oc_X(D))-\{0\})\rightarrow \mathbb{Z}^n.
$$
as follows. First set $\mu_1=\mu_1(s)={\rm ord}_{Y_1}(s)$. Dividing $s$ by a local equation of $Y_1$,  we obtain a section
$$
\widetilde{s}_1\in H^0(X,\oc_X(D-\mu_1Y_1))
$$
that does not vanish identically along $Y_1$. We restrict $\widetilde{s}_1$ on $Y_1$ to get a non-zero section
$$
s_1\in H^0(Y_1,\oc_{Y_1}(D-\mu_1Y_1)),
$$
then we write
$\mu_2(s)=\ord_{Y_2}(s_1)$,
and continue in this fashion to define the remaining integers $\mu_i(s)$. The image of the
function $\mu$ in $\mathbb{Z}^n$ is denoted by $\mu(D)$. With this in hand, we define the \textit{Okounkov body of D with respect to the fixed flag $Y_{\bullet}$} to be
$$
\Delta(D)=\Delta_{Y_{\bullet}}(D)= \text{closed convex hull}\left(\displaystyle\bigcup_{m\geq 1}\frac{1}{m}\cdot \mu(mD)\right)\subseteq \mathbb{R}^n.
$$

According to the open question raised in the final part of \cite{LM09}, it is quite natural to wonder whether one can  construct ``arithmetic Okounkov bodies" for an arbitrary pseudo-effective (1,1)-class $\alpha$ on a K\"ahler manifold, and realize the volumes of these classes by convex bodies as well. In our paper, using positive currents in a natural way, we give a construction of a convex body $\Delta(\alpha)$ associated to such a class $\alpha$, and show that this newly defined convex body coincides with the Okounkov body when $\alpha\in {\rm NS}_{\mathbb{R}}(X)$.

\begin{thm}
\label{equivalent}
Let $X$ be a smooth projective variety of dimension $n$, $L$ be a big line bundle on $X$ and $Y_{\bullet}$ be a fixed admissible flag. Then we have $$\Delta(c_1(L))=\Delta(L)=\overline{\bigcup_{m=1}^{\infty} \frac{1}{m}\nu(mL)}.$$ Moreover, in the definition of Okounkov body $\Delta(L)$, it suffices to take the closure of the set of normalized valuation vectors instead of the closure of the convex hull.
\end{thm}

By Theorem \ref{equivalent}, we know that our definition of the Okounkov body for any pseudo-efffective class could be treated as a generalization of the original Okounkov body. A very interesting problem is to find out exactly
which points in the Okounkov body $\Delta(L)$ are given by valuations of sections. This is expressed by saying that a rational point of $\Delta(L)$ is
``valuative". By Theorem \ref{equivalent} we can give some partial answers to this question which have been given in  \cite{KL14} in the case of surfaces. 
\begin{cor}
\label{valuation}
Let $X$ be a projective variety of dimension $n$ and $Y_{\bullet}$ be an admissible flag. If $L$ is a big line bundle, then any rational point in ${\rm int}(\Delta(L))$ is a valuative point.
\end{cor}

It is quite natural to wonder whether our newly defined convex body for big classes behaves similarly as the original Okounkov body. In the situation of complex surfaces, we give an affirmative answer to the question raised in \cite{LM09}, as follows:

\begin{thm}
\label{Okounkov}
Let $X$ be a compact K\"ahler surface, $\alpha\in H^{1,1}(X,\mathbb{R})$ be a big class. If $C$ is an irreducible divisor of $X$, there are piecewise linear continuous functions
$$ f\ ,\ g\ :\ [a,s]\mapsto \mathbb{R}_+$$
with $f$ convex, $g$ concave, and $f\leq g$, such that $\Delta(\alpha)\subset \mathbb{R}^2$ is the region bounded by the graphs of $f$ and $g$:
$$
\Delta(\alpha)=\{(t,y)\in \mathbb{R}^2\mid a\leq t \leq s, \mbox{and}\ f(t)\leq y\leq g(t)\}.
$$
Here $\Delta(\alpha)$ is the generalized Okounkov body with respect to the fixed flag
$$
X\supseteq C\supseteq \{x\},
$$
and $s=\sup\{t>0\mid \alpha-tC {\rm\ is\ big}\}$. If $C$ is nef, $a=0$ and $f(t)$ is increasing; otherwise, $a=\sup\{t>0\mid C\subseteq E_{nK}(\alpha-tC)\}$, where  $E_{nK}:=\bigcap_TE_+(T)$ for $T$ ranging among the K\"ahler currents in $\alpha$, which is the non-K\"ahler locus.
Moreover, $\Delta(\alpha)$ is a finite polygon whose number of vertices is bounded by $2\rho(X)+2$, where $\rho(X)$ is the Picard number of $X$, and
$$
\vol_X(\alpha)=2\vol_{\mathbb{R}^2}(\Delta(\alpha)).
$$
\end{thm}

In \cite{LM09}, it was asked whether the Okounkov body of a divisor on a complex surface could be an infinite polygon. In \cite{KLM10}, it was shown that the Okounkov body is always a finite polygon. Here we give an explicit description for the ``finiteness" of the polygons appearing as generalized Okounkov bodies of big classes, and conclude that it also holds for the original Okounkov bodies by Theorem \ref{equivalent}.

As one might suspect from the construction of Okounkov bodies, the Euclidean volume of $\Delta(D)$ has a strong connection with the growth of the groups $H^0(X,\oc_X(mD))$. In \cite{LM09}, the following precise relations were shown:
\begin{equation}
\label{volume equal}
n!\cdot\vol_{\mathbb{R}^n}(\Delta(D))= \vol_X(D):=\lim\limits_{k\rightarrow\infty}\frac{n!}{k^n}h^0(X,\oc_X(kD)).
\end{equation}
The proof of (\ref{volume equal}) relies on  properties of sub-semigroups of $\mathbb{N}^{n+1}$ constructed from the graded linear series $\{H^0(X,\oc_X(mD))\}_{m\geq0}$. However, when $\alpha$ is a big class which does not belong to ${\rm NS}_{\mathbb{R}}(X)$, there are no such algebraic objects which correspond to $\vol_X(\alpha)$, and we only have the following analytic definition due to Boucksom:
$$\vol_X(\alpha):= \sup_{T}\int_{X} T_{ac}^n,$$
where $T$ ranges among all positive $(1,1)$-currents. Therefore, it is quite natural to propose the following conjecture:
\begin{conj}
Let $X$ be a compact K\"ahler manifold of dimension $n$. For any big class $\alpha\in H^{1,1}(X,\mathbb{R})$, we have
$$
\vol_{\mathbb{R}^n}(\Delta(\alpha))=\frac{1}{n!}\cdot \vol_X(\alpha).
$$
\end{conj}

In Theorem \ref{Okounkov}, we prove this conjecture in dimension 2. Our method is to relate the Euclidean volume of the slice of the generalized Okounkov body to the differential of the volume of the big class. We prove the following differentiability formula for volumes of big classses.

\begin{thm}[Differentiability of volumes]
\label{main differential}
Let $X$ be a compact K\"ahler surface and $\alpha$ be a big class. If $\beta$ is a nef class or $\beta=\{C\}$ where $C$ is an irreducible curve, we have
$$
\left.\frac{d}{dt}\right|_{t=0}\vol_X(\alpha+t\beta)=2Z(\alpha)\cdot \beta,
$$
where $Z(\alpha)$ is the divisorial Zariski decomposition of $\alpha$ defined in Section \ref{divisorial}.
\end{thm}

A direct corollary of this formula is the  \textit{transcendental Morse inequality}:

\begin{thm}\it
\label{morse 2}
Let $X$ be a compact K\"ahler surface. If $\alpha$ and $\beta$ are nef classes satisfying the inequality ${\alpha}^2- 2\alpha \cdot \beta>0$, then $\alpha-\beta$ is big and $\vol_X(\alpha-\beta) \geq {\alpha}^2- 2\alpha \cdot \beta $.
\end{thm}

In higher dimension, we also have a differentiability formula for big classes on some special K\"ahler manifolds.
\begin{thm}\it
\label{morse high}
Let $X$ be a compact K\"ahler manifold of dimension $n$ on which the modified nef cone $\mathcal{MN}$ coincides with the nef cone $\mathcal{N}$. If $\alpha\in H^{1,1}(X,\mathbb{R})$ is a big class, $\beta \in H^{1,1}(X,\mathbb{R})$ is a nef class, then \begin{equation}
\label{volume formula}
\vol_X(\alpha+\beta) = \vol_X(\alpha)+n\int_{0}^{1} Z(\alpha+t\beta)^{n-1}\cdot \beta\ \mathrm{d}t.
\end{equation}
As a consequence, $\vol_X(\alpha+t\beta)$ is $\mathcal{C}^1$ for $t\in \mathbb{R}^+$ and we have
\begin{equation}
\left.\frac{d}{dt}\right|_{t=t_0}\vol_X(\alpha+t\beta)=nZ(\alpha+t_0\beta)^{n-1}\cdot \beta
\end{equation}
\label{diff higher}
for $t_0\geq0$.
\end{thm}

Finally, we study the generalized Okounkov bodies for pseudo-effective classes in K\"ahler surfaces. We summerize our results as follows
\begin{thm}\it
\label{okounkov psf}
Let $X$ be a K\"ahler surface, $\alpha$ be any pseudo-effective but not big class,
\begin{enumerate}[\upshape (i)]
\item if the numerical dimension $n(\alpha)=0$, then for any irreducible curve $C$ which is not contained in the negative part $N(\alpha)$, we have the generalized Okounkov body 
$$\Delta_{(C,x)}(\alpha)=0\times \nu_x(N(\alpha)|_C),$$
where $\nu_x(N(\alpha)|_C)=\nu(N(\alpha)|_C,x)$ is the Lelong number of $N(\alpha)$ at $x$;
\item if $n(\alpha)=1$, then for any irreducible curve $C$ satisfying $Z(\alpha)\cdot C>0$, we have
$$\Delta_{(C,x)}(\alpha)=0\times [\nu_x(N(\alpha)|_C),\nu_x(N(\alpha)|_C)+Z(\alpha)\cdot C]\nonumber.$$
\end{enumerate}
In particular, the numerical dimension determines the dimension of the generalized Okounkov body.
\end{thm}

\section{Technical preliminaries}
\label{preliminaries}

\subsection{Siu decomposition}
\label{Siu}
Let $T$ be a closed positive current of bidegree $(p,p)$ on a complex manifold $X$. We denote by $\nu(T,x)$ its Lelong number at a point $x\in X$. For any $c>0$, the Lelong upperlevel sets are defined by
$$
E_c(T):=\{x\in X, \nu(T,x)\geq c\}.
$$
In \cite{Siu74}, Siu proved that $E_c(T)$ is an analytic subset of $X$, of codimension at least $p$. Moreover, $T$ can be written as a convergent series of closed positive currents
$$
T=\displaystyle\sum_{k=1}^{+\infty}\nu(T,Z_k) [Z_k]+R
$$
where $[Z_k]$ is a current of integration over an irreducible analytic set of dimension $p$, and $R$ is a residual current with the property that $\dim E_c(R)<p$ for every $c>0$. This decomposition is locally and globally unique: the sets $Z_k$ are precisely the $p$-dimensional components occurring in the upperlevel sets $E_c(T)$, and $\nu(T,Z_k):=\inf\{\nu(T,x)|x\in Z_k\}$ is the generic Lelong number of $T$ along $Z_k$.

\subsection{Currents with analytic singularities}
\label{analytic singularity}

A closed positive (1,1) current $T$ on a compact complex manifold $X$ is said to have analytic  (resp. algebraic) singularities along a subscheme $V(\mathcal{I})$ defined by an ideal $\mathcal{I}$ if there exists some $c\in \mathbb{R}_{>0}$ (resp. $\mathbb{Q}_{>0}$) such that locally we have
$$
T=\frac{c}{2}dd^c\log(|f_1|^2+\ldots+|f_k|^2)+dd^cv
$$
where $f_1,\ldots,f_k$ are local generators of $\mathcal{I}$ and $v\in L^{\infty}_{\rm loc}$ (resp. and additionally, $X$ and $V({\cal I})$ are algebraic). Moreover, if $v$ is smooth, $T$ will be said to have mild analytic singularities. In these situations, we call the sum $\sum\nu(T,D)D$ which appears in the Siu decomposition of $T$ the divisorial part of $T$. Using the Lelong-Poincar\'e formula, it is straightforward to check that the divisorial part $\sum\nu(T,D)D$ of a closed (1,1)-current $T$ with analytic singularities along the subscheme $V(\mathcal{I})$ is just the divisorial part of $V(\mathcal{I})$, times the constant $c>0$ appearing in the definition of analytic singularities. The residual part $R$ has analytic singularities in codimension at least~$2$. If we denote $E_+(T):=\{x\in X|\nu(T,x)>0\}$, then $E_+(T)$ is exactly the support of $V(\mathcal{I})$. Moreover, if $V\not\subseteq E_+(T)$ for some smooth variety~$V$, $T|_V:=\frac{c}{2}dd^clog(|f_1|^2+\ldots+|f_k|^2)|_V+dd^cv|_V$ is well defined, for $|f_1|^2+\ldots+|f_k|^2$ and $v$ are not identically equal to $-\infty$ on $V$. It is easy to check that this definition does not depend on the choice of the local potential of $T$.

\begin{defin}[Non-K\"ahler locus]
If $\alpha\in H^{1,1}_{{\partial}\overline{\partial}}(X,\mathbb{R})$ is a big class, we define its \emph{non-K\"ahler locus} as $E_{nK}:=\bigcap_TE_+(T)$ for $T$ ranging among the K\"ahler currents in $\alpha$.
\end{defin}

We will usually use the following theorem due to Collins and Tosatti.

\begin{thm}[\cite{CT13}]\it
\label{Tosatti}
Let $X$ be a compact K\"ahler manifold of dimension $n$. Given a nef and big class $\alpha$, we define a subset of $X$ which measures
its non-K\"ahlerianity, namely the null locus
$$\text{\rm Null}(\alpha):=\bigcup_{\int_{V}^{} \alpha^{\text{dim}V}=0}V,$$
where the union is taken over all positive dimensional irreducible
analytic subvarieties of X. Then we have $$\text{\rm Null}(\alpha)=E_{nK}(\alpha).$$
\end{thm}

\subsection{Regularization of currents}
We will need Demailly's regularization theorem for closed (1,1)-currents, which enables us to approximate a given current by currents with analytic singularities, with a loss of positivity that is arbitrary small.  In particular, we could approximate a K\"ahler current $T$ inside its cohomology class by K\"ahler currents $T_k$ with algebraic singularities, with a good control of the singularities. A big class therefore contains plenty of K\"ahler currents with analytic singularities.
\begin{thm}\it
\label{Demailly}
Let $T$ be a closed almost positive (1,1)-current on a compact  complex manifold $X$, and fix a Hermitian form $\omega$. Suppose that $T\geq \gamma$ for some real (1,1)-form $\gamma$ on $X$. Then there exists a sequence $T_k$ of currents with algebraic singularities in the cohomology class $\{T\}$ which converges weakly to T, such that $T_k\geq \gamma-\epsilon_k \omega$ for some sequence $\epsilon_k>0$ decreasing to 0, and $\nu(T_k,x)$ increases to $\nu(T,x)$ uniformly with respect to $x\in X$.
\end{thm}

\subsection{Currents with minimal singularities}
Let $T_1=\theta_1+dd^c\varphi_1$ and $T_2=\theta_2+dd^c\varphi_2$ be two closed almost positive (1,1)-currents on $X$, where $\theta_i$ are smooth forms and $\varphi_i$ are almost pluri-subharmonic functions, we say that $T_1$ is less singular than $T_2$ (write $T_1\preceq T_2$) if we have $\varphi_2\leq \varphi_1+C$ for some constant $C$.

Let $\alpha$ be a class in $H^{1,1}_{{\partial}\overline{\partial}}(X,\mathbb{R})$ and $\gamma$ be a smooth real (1,1)-form, we denote by $\alpha[\gamma]$ the set of closed almost positive (1,1)-currents $T\in \alpha$ with $T\geq\gamma$. Since the set of potentials of such currents is stable by taking a supremum, we conclude by standard pluripotential theory that there exists a closed almost positive (1,1)-current $T_{\min,\gamma}\in \alpha[\gamma]$ which has minimal singularities in $\alpha[\gamma]$. $T_{\min,\gamma}$ is well defined modulo $dd^cL^{\infty}$. For each $\epsilon>0$, denote by $T_{\min,\epsilon}=T_{\min,\epsilon}(\alpha)$ a current with minimal singularities in $\alpha[-\omega]$, where $\omega$ is some reference Hermitian form. The minimal multiplicity at $x\in X$ of the pseudo-effective class $\alpha\in H^{1,1}_{{\partial}\overline{\partial}}(X,\mathbb{R})$
 is defined as
$$
\nu(\alpha,x):=\sup_{\epsilon>0}\nu(T_{\min,\epsilon},x).
$$
For a prime divisor $D$, we define  the generic minimal multiplicity of $\alpha$ along $D$ as
$$
\nu(\alpha,D):=\inf\{\nu(\alpha,x)|x\in D)\}.
$$
We then have $\nu(\alpha,D)=\sup_{\epsilon>0}\nu(T_{\min,\epsilon},D)$.

\subsection{Lebesgue decomposition}
A current $T$ can be locally seen as a form with distribution coefficients.  When $T$ is positive, the distributions are positive measures which admit a Lebesgue decomposition into an absolutely continuous part (with respect to the Lebesgue measure on $X$) and a singular part. Therefore we obtain the decomposition $T=T_{\rm ac}+T_{\rm sing}$, with $T_{\rm ac}$ (resp. $T_{\rm sing}$) globally determined thanks to the uniqueness of the Lebesgue decomposition.

Now we assume that $T$ is a (1,1)-current. The absolutely continuous part $T_{\rm ac}$ is considered as a (1,1)-form with $L^1_{\rm loc}$ coefficients, and more generally we have $T_{\rm ac} \geq \gamma$ whenever $T\geq \gamma$ for some real smooth real form $\gamma$. Thus we can define the product $T^k_{\rm ac}$ of $T_{\rm ac}$ almost everywhere. This yields a positive Borel $(k,k)$-form.

\subsection{Modified nef cone and divisorial Zariski decomposition}
\label{divisorial}
In this subsection, we collect some definitions and properties of the modified nef cone and divisorial Zariski decomposition. See \cite{Bou04} for more details.
\begin{defin}
Let $X$ be compact complex manifold, and $\omega$ be some reference Hermitian form. Let $\alpha$ be a class in $H^{1,1}_{{\partial}\overline{\partial}}(X,\mathbb{R})$.
\begin{enumerate}[\upshape (i)]
\item $\alpha$ is said to be a \emph{modified K\"ahler class} iff it contains a K\"ahler current $T$ with $\nu(T,D)=0$ for all prime divisors $D$ in $X$.
\item $\alpha$ is said to be a \emph{modified nef class} iff, for every $\epsilon>0$, there exists a closed (1,1)-current $T_\epsilon\geq -\epsilon\omega$ and $\nu(T_\epsilon,D)=0$ for every prime $D$.
\end{enumerate}
\end{defin}

\begin{rem}
\label{mf=f}
The modified nef cone $\mathcal{MN}$ is a closed convex cone which contains the nef cone $\mathcal{N}$. When $X$ is a K\"ahler manifold, $\mathcal{MN}$ is just the interior of the modified K\"ahler cone $\mathcal{MK}$.
\end{rem}

\begin{rem}
\label{nef property}
For a complex surface, the K\"ahler (nef) cone  and the modified K\"ahler (modified nef) cone coincide. Indeed, analytic singularities in codimension 2 of a K\"ahler current $T$ are just  isolated points. Therefore the class $\{T\}$ is a K\"ahler class.
\end{rem}

\begin{defin}[Divisorial Zariski decomposition]
The \emph{negative part} of a pseudo-effective class $\alpha\in H^{1,1}_{{\partial}\overline{\partial}}(X,\mathbb{R})$ is defined as $N(\alpha):=\sum\nu(\alpha,D)D$. The \emph{Zariski projection} of $\alpha$ is $Z(\alpha):=\alpha-\{N(\alpha)\}$. We call the decomposition $\alpha=Z(\alpha)+\{N(\alpha)\}$ the \emph{divisorial Zariski decomposition of $\alpha$}.
\end{defin}

\begin{rem}
\label{bijection}
We claim that the volume of $Z(\alpha)$ is equal to the volume of $\alpha$. Indeed, if $T$ is a positive current in $\alpha$, then we have $T\geq N(\alpha)$ since $T\in \alpha[-\epsilon\omega]$ for each $\epsilon>0$ and we conclude that $T\mapsto T-N(\alpha)$ is a bijection between the positive currents in $\alpha$ and those in $Z(\alpha)$. Furthermore, we notice that $(T-N(\alpha))_{\rm ac}=T_{\rm ac}$, and thus by the definition of volume of the pseudo-effective classes we conclude that $\vol_X(\alpha)=\vol_X(Z(\alpha))$.
\end{rem}

\begin{defin}[Exceptional divisors]
\begin{enumerate}[\upshape (i)]
\item A family $D_1,\ldots\ ,D_q$ of prime divisors is said to be an \emph{exceptional family} iff the convex cone generated by their cohomology classes meets the modified nef cone at 0 only.
\item An effective $\mathbb{R}$-divisor $E$ is said to be \emph{exceptional} iff its prime components constitute an exceptional family.
\end{enumerate}
\end{defin}

We have the following properties of exceptional divisors:
\begin{thm}\it
\label{property exceptional}
\begin{enumerate}[\upshape (i)]
\item An effective $\mathbb{R}$-divisor E is exceptional iff $Z({E})=0$.
\item If E is an exceptional effective $\mathbb{R}$-divisor, we have $E=N(\{E\})$.
\item If $D_1,\ldots,D_q$ is an exceptional family of primes, then their classes $\{D_1\},\ldots,\{D_q\}$ are linearly independent in ${\rm NS}_{\mathbb{R}}(X)\subset H^{1,1}(X,\mathbb{R})$. In particular, the length of the exceptional families of primes is uniformly bounded by the Picard number $\rho(X)$.
\item Let $X$ be a surface, a family $D_1,\ldots,D_r$ of prime divisors is exceptional iff its intersection matrix $(D_i\cdot D_j)$ is negative definite.
\end{enumerate}
\end{thm}

In this paper, we need the following properties of the modified nef cone $\mathcal{MN}$ and the divisorial Zariski decomposition due to Boucksom (ref. \cite{Bou04}). We state these properties without proofs.

\begin{thm}\it
\label{modified big}
Let $\alpha\in H^{1,1}(X,\mathbb{R})$ be a pseudo-effective class. Then we have:
\begin{enumerate}[\upshape (i)]
\item Its Zariski projection $Z(\alpha)$ is a modified nef class.
\item $Z(\alpha)=\alpha$ iff $\alpha$ is modified nef.
\item $Z(\alpha)$ is big iff $\alpha$ is.
\end{enumerate}
\end{thm}

\begin{rem}
Let $X$ be a complex K\"ahler surface. For a big class $\alpha\in H^{1,1}(X,\mathbb{R})$, $Z(\alpha)$ is a big and modified nef class. By Remark \ref{mf=f}, any modified nef class is nef, it follows that $Z(\alpha)$ is big and nef.
\end{rem}

\begin{thm}\it
\label{Continuous}
\begin{enumerate}[\upshape (i)]
\item The map $\alpha \mapsto N(\alpha)$ is convex and homogeneous on pseudo-effective class cone $ \mathcal{E}$. It is continuous on the interior of $ \mathcal{E}$.
\item The Zariski projection $Z:\mathcal{E}\rightarrow \mathcal{MN} $ is concave and homogeneous. It is continuous on the interior of $\mathcal{E}$.
\end{enumerate}
\end{thm}

\begin{thm}\it
\label{contain exceptional}
Let p be a big and modified nef class. Then the primes $D_1,\ldots,D_q$ contained in the non-K\"ahler locus $E_{nK}(p)$ form an exceptional family $A$, and the fiber of Z over p is the simplicial cone $Z^{-1}(p)=p+V_+(A)$, where $V_+(A):=\sum_{D\in A}\mathbb{R}_+\{D\}$.
\end{thm}

\begin{thm}\it
\label{orthogonal}
Let $X$ be a compact surface. If $\alpha\in H^{1,1}(X,\mathbb{R})$ is a pseudo-effective class, its divisorial Zariski decomposition $\alpha=Z(\alpha)+\{N(\alpha)\}$ is the unique orthogonal decomposition of $\alpha$ with respect to the non-degenerate quadratic form $q(\alpha):=\int\alpha^2$ into the sum of a modified nef class and the class of an exceptional effective $\mathbb{R}$-divisor.
\end{thm}

\begin{rem}
Let $X$ be a surface, $\alpha$ is the class of an effective $\mathbb{Q}$-divisor $D$ on a projective surface, the divisorial Zariski decomposition of $\alpha$ is just the original Zariski decomposition of $D$.
\end{rem}

\section{Transcendental Morse inequality}

\subsection{Proof of the transcendental Morse inequality for complex surfaces}

The main goal of this section is to prove the differentiability of the volume function and the transcendental Morse inequality for complex surfaces. In fact, in the next subsection we will give a more general method to prove the transcendental Morse inequality for K\"ahler manifolds on which modified nef cones $\mathcal{MN}$ coincide with the nef cones; this includes complex surfaces. However, since the methods and results here are very special in studying generalized Okounkov bodies, we will treat complex surface and higher dimensional K\"ahler manifolds separately. Throughout this subsection, if not specially mentioned, $X$ will stand for a complex K\"ahler surface. We denote by $q(\alpha):=\int \alpha^2$ the quadratic form on $H^{1,1}(X,\mathbb{R})$. By the Hodge index theorem, $(H^{1,1}(X,\mathbb{R}),q)$ has signature $(1,h^{1,1}(X)-1)$. The open cone $\{\alpha\in H^{1,1}(X,\mathbb{R})|q(\alpha)>0\}$ has thus two connected components which are convex cones, and we denote by $\mathcal{P}$ the component containing the K\"ahler cone $\mathcal{K}$.
\begin{lem}\it
\label{component}
Let $X$ be a compact K\"ahler manifold of dimension $n$. If $\alpha\in H^{1,1}(X,\mathbb{R})$ is a big class, $\beta \in H^{1,1}(X,\mathbb{R})$ is a nef class, then $N(\alpha+t\beta)\leq N(\alpha)$ as effective $\mathbb{R}$-divisors for $t\geq 0$. Furthermore, when t is small enough, the prime components of $N(\alpha+t\beta)$ will be the same as those of $N(\alpha)$.
\end{lem}
\begin{proof}
Since $\beta$ is nef, by Theorem \ref{Continuous}, we have
$$
N(\alpha+t\beta)\leq N(\alpha)+tN(\beta)=N(\alpha).
$$
Since the map $\alpha \mapsto N(\alpha)$ is convex on pseudo-effective class cone $ \mathcal{E}$, it is continuous on the interior of $ \mathcal{E}$, and thus the theorem follows.
\end{proof}

\begin{thm}\it
\label{differential}
If $\alpha\in H^{1,1}(X,\mathbb{R})$ is a big class and $\beta \in H^{1,1}(X,\mathbb{R})$ is a nef class, then
\begin{equation}
\label{differential formula}
\left.\frac{d}{dt}\right|_{t=0}\vol_X(\alpha+t\beta)=2Z(\alpha)\cdot \beta
\end{equation}
\end{thm}
\begin{proof}
By Lemma \ref{component}, there exists an $\epsilon>0$ such that when $0\leq t<\epsilon$, we can write $N(\alpha+t\beta)=\sum_{i=1}^{r}a_i(t)N_i$, where $0< a_i(t)\leq a_i(0)=:a_i$, and each $a_i(t)$ is a continuous and decreasing function with respect to $t$. According to the orthogonal property of divisorial Zariski decomposition (ref. Theorem \ref{orthogonal}),    $Z(\alpha+t\beta)\cdot N(\alpha+t\beta)=0$ for $t\geq 0$. Since $Z(\alpha+t\beta)$ is modified nef and thus nef (by Remark \ref{nef property}), we have $Z(\alpha+t\beta)\cdot N_i\geq 0$ for every $i$. When $0\leq t<\epsilon$, we have $a_i(t)>0$ for $i=1,\ldots,r$, therefore, $Z(\alpha+t\beta)$ is orthogonal to each $\{N_i\}$ with respect to $q$. We denote by $V\subset H^{1,1}(X,\mathbb{R})$ the finite vector space spanned by $\{N_1\},\ldots,\{N_r\}$, by $V^{\bot}$ the orthogonal space of V with respect to $q$. Thus $\alpha+t\beta=Z(\alpha+t\beta)+\sum_{i=1}^{r}a_i(t)\{N_i\}$ is the decomposition in the direct sum $V^{\bot} \oplus V$. We decompose  $\beta=\beta^{\bot}+\beta_0$ in the direct sum $V^{\bot}\oplus V$, and we have
\begin{gather}
Z(\alpha+t\beta)=Z(\alpha)+t\beta^{\bot},\nonumber\\
\sum_{i=1}^{r}a_i(t)\{N_i\}=\sum_{i=1}^{r}a_i\{N_i\}+t\beta_0.\nonumber
\end{gather}
Since $\vol_X(\alpha+t\beta)=\vol_X(Z(\alpha+t\beta))= Z(\alpha+t\beta)^2$ (by Remark \ref{bijection}),
it is easy to deduce that
$$
\left.\frac{d}{dt}\right|_{t=0}\vol_X(\alpha+t\beta)=2Z(\alpha)\cdot \beta^{\bot}=2Z(\alpha)\cdot \beta.
$$
The last equality follows from $\beta_0\in V$ and $Z(\alpha)\in V^{\bot}$. We get the first half of Theorem \ref{main differential}.
\end{proof}

To prove the transcendental Morse inequality for complex surfaces, we will need a criterion for bigness of a class:
\begin{thm}\it
\label{criterion}
Let $\alpha$ and $\beta$ be two nef classes such that $\alpha^2-2\alpha\cdot \beta>0$, then $\alpha-\beta$ is a big class.
\end{thm}

\begin{proof}
We denote by $\mathcal{P}$ the connected component of the open cone $\{\alpha\in H^{1,1}(X,\mathbb{R})\mid q(\alpha)>0\}$ containing the K\"ahler cone $\mathcal{K}$, then $\mathcal{P}\subset \mathcal{E}^0$. As a consequence of the Nakai-Moishezon criterion for surfaces (ref. \cite{Lam99}), we know that, if $\gamma$ is a real (1,1)-class with $\gamma^2>0$, then $\gamma$ or $-\gamma$ is big. Since $\alpha$ and $\beta$ are both nef, we have that $(\alpha-t\beta)^2>0$ for $0\leq t\leq 1$. This means that $\alpha-t\beta$ is contained in some component of the open cone $\{\alpha\in H^{1,1}(X,\mathbb{R})|q(\alpha)>0\}$. But since $\alpha$ is big, $\alpha-t\beta$ is contained in $\mathcal{P}\subset \mathcal{B}$, and \textit{a fortiori} $\alpha-\beta$ is.
\end{proof}

Now we are ready to prove the transcendental Morse inequality for complex surfaces.
\begin{proof}[Proof of Theorem \ref{morse 2}]
By Theorem \ref{criterion}, when ${\alpha}^2- 2\alpha \cdot \beta>0$, the cohomology class $\alpha-\beta$ is big. By the differentiability formula  (\ref{differential formula}), we have
$$
\vol_X(\alpha-\beta)=\alpha^2-2\int_{0}^{1} Z(\alpha-t\beta)\cdot \beta\ dt.
$$
Since the Zariski projection $Z:\mathcal{E}\rightarrow \mathcal{MN} $ is concave and homogeneous by Theorem \ref{Continuous}, we have $$\alpha=Z(\alpha)\geq Z(\alpha-t\beta)+Z(t\beta)\geq Z(\alpha-t\beta).$$
Since $\beta$ is nef, we have
$$
\alpha\cdot \beta\geq Z(\alpha-t\beta)\cdot \beta,
$$
and thus
$$
\vol_X(\alpha-\beta) \geq {\alpha}^2- 2\alpha \cdot \beta.
$$
\end{proof}

In the last part of this subsection, we prove the second half of Theorem \ref{main differential}.
\begin{thm}\it
\label{differential 2}
Let $\alpha\in H^{1,1}(X,\mathbb{R})$ be a big class and $C$ be an irreducible divisor, then
\begin{equation}
\label{differential formula 2}
\left.\frac{d}{dt}\right|_{t=0}\vol_X(\alpha+tC)=2Z(\alpha)\cdot C.
\end{equation}
\end{thm}
\begin{proof}
It suffices to prove the theorem for $C$ not nef. Thus we have $C^2<0$. Write $N(\alpha)=\sum_{i=1}^{r}a_iN_i$, where each $N_i$ is prime divisor. If $C\subseteq E_{nK}(Z(\alpha))$, we deduce that $Z(\alpha)\cdot C=0$ by Theorem \ref{Tosatti}, and $\{C,N_1,\ldots,N_r\}$ forms an exceptional family by Theorem \ref{contain exceptional}. Thus we have
$$
Z(\alpha+tC)=Z(\alpha),
$$
and
$$
N(\alpha+tC)=N(\alpha)+tC
$$
for $t\geq0$. The theorem is thus proved in this case.

From now on we assume $C\not\subseteq E_{nK}(Z(\alpha))$, thus we have $Z(\alpha)\cdot C> 0$ and $C\not\subseteq {\rm Supp}(N(\alpha))$. We define
\begin{equation}
\left(
  \begin{array}{c}
    b_1\\
    \vdots\\
    b_r\\
  \end{array}
\right)=
-S^{-1}\cdot\left(
  \begin{array}{c}
    C\cdot N_1\\
    \vdots\\
    C\cdot N_r\\
  \end{array}
\right)\nonumber,
\end{equation}
where $S=(s_{ij})$ denotes the intersection matrix of $\{N_1,\ldots, N_r\}$. By Theorem \ref{orthogonal} we know that $S$ is negative definite satisfying $s_{ij}\geq 0$ for all $i\neq j$. We claim that $Z(\alpha)+t(\{C\}+\sum_{i=1}^{r}b_i\{N_i\})$ is big and nef if $0\leq t< -\frac{Z(\alpha)\cdot C}{C^2}$. We need the following lemma from \cite{BKS03} to prove our claim.

\begin{lem}
\label{BKS}
Let $A$ be a negative definite $r\times r$-matrix over the reals such that $a_{ij}\geq 0$ for all $i\neq j$. Then all entries of the inverse matrix $A^{-1}$ are $\leq0$.
\end{lem} 
By Lemma \ref{BKS} we know that all entries of $S^{-1}$ are $\leq0$, thus $b_j\geq0$ for all $1\leq j\leq r$ and we get the bigness of $Z(\alpha)+t(\{C\}+\sum_{i=1}^{r}b_i\{N_i\})$. By the construction of $b_j$, we have
$$
(Z(\alpha)+t(\{C\}+\sum_{i=1}^{r}b_i\{N_i\}))\cdot N_j=0
$$
for $1\leq j\leq r$, and 
$$
(Z(\alpha)+t(\{C\}+\sum_{i=1}^{r}b_i\{N_i\}))\cdot C>0
$$
for $0\leq t<-\frac{Z(\alpha)\cdot C}{C^2}$. Thus we have the nefness and our claim follows. Since the divisorial Zariski decomposition is orthogonal and unique (see Theorem \ref{orthogonal}), we conclude that 
\begin{gather}
\label{negative}
N(\alpha+t\{C\})=\sum_{i=1}^{r}(a_i-tb_i)N_i,\\
Z(\alpha+t\{C\})=Z(\alpha)+t\{C\}+\sum_{i=1}^{r}tb_i\{N_i\},
\end{gather}
for $t$ small enough. Since $\vol_X(\alpha+tC)=Z(\alpha+t\{C\})^2$, we have thus also obtained formula (\ref{differential formula 2}) in this case.

\end{proof}

\subsection{Transcendental Morse inequality for some special K\"ahler manifolds}
One can modify the proof of Theorem \ref{morse 2} a little bit, to extend the transcendental Morse inequality to K\"ahler manifolds whose modified nef cone $\mathcal{MN}$ coincides with the nef cone $\mathcal{N}$. In this subsection, we assume $X$ to be a compact K\"ahler manifold of dimension $n$ which satisfies this condition.

\begin{lem}\it
\label{orthogonal higher}
If $\alpha\in \mathcal{E}^\circ$, then the divisorial Zariski decomposition $\alpha =Z(\alpha)+N(\alpha)$ is such that
$$Z(\alpha)^{n-1}\cdot N(\alpha)=0.$$
\end{lem}
\begin{rem}
Lemma \ref{orthogonal higher} is very similar to the Corollary 4.5 in \cite{BDPP13}: If $\alpha\in \mathcal{E}_{\rm NS}$, then the divisorial Zariski decomposition  $\alpha= Z(\alpha)+N(\alpha)$ is such that  $\langle Z(\alpha)^{n-1} \rangle \cdot N(\alpha)=0$. However, the proof of \cite{BDPP13} is based on the orthogonal estimate for divisorial Zariski decomposition of $\mathcal{E}_{\rm NS}$, which is still a conjecture for $\alpha\in \mathcal{E}$. Here we will use Theorem \ref{Tosatti} to prove this lemma directly.
\end{rem}

\begin{proof}[Proof of Lemma \ref{orthogonal higher}]
By Theorem \ref{modified big}, if $\alpha$ is big, then $Z(\alpha)$ is big and modified nef, thus nef by the assumption for $X$. By Theorem \ref{contain exceptional}, the primes $D_1,\ldots,D_q$ contained in the non-K\"ahler locus $E_{nK}(Z(\alpha))$ form an exceptional family, and $\alpha=Z(\alpha)+\sum_{i=1}^{r} a_iD_i$ for $a_i\geq0$ . Since $\text{\rm Null}(Z(\alpha))=E_{nK}(Z(\alpha))$ by Theorem \ref{Tosatti}, we have  $Z(\alpha)^{n-1}\cdot D_i=0$ for each $i$, and thus $Z(\alpha)^{n-1}\cdot N(\alpha)=0$. The lemma is proved.
\end{proof}

\begin{proof}[Proof of Theorem \ref{morse high}]
By Lemma \ref{component}, there exists $\epsilon>0$ such that the prime components of $N(\alpha+t\beta)$ will be the same when $0\leq t\leq\epsilon$. Moreover if we denote $N(\alpha+t\beta)=\sum_{i=1}^{r}a_i(t)N_i$, then each $a_i(t)$ is continuous and decreasing satisfying $a_i(t)>0$. By Lemma \ref{orthogonal higher}, we have $$Z(\alpha+t\beta)^{n-1}\cdot N(\alpha+t\beta)=\sum_{i=1}^{r}a_i(t)Z(\alpha+t\beta)^{n-1}\cdot N_i=0.$$
Since $Z(\alpha+t\beta)$ is modified nef thus nef, we deduce that $Z(\alpha+t\beta)^{n-1}\cdot N_i=0$ for $0\leq t\leq\epsilon$ and $i=1,\ldots,r$.

Since $a_i(t)$ is continuous and decreasing, it is almost everywhere differentiable. Thus $Z(\alpha+t\beta)=\alpha+t\beta-\sum_{i=1}^{r}a_i(t)N_i$ is an a.e. differentiable and continuous curves in the finite dimensional space $H^{1,1}(X,\mathbb{R})$ parametrized by $t$. Meanwhile, since $\alpha\mapsto \alpha^n$ is a quadratic form (possibly degenerate) in $H^{1,1}(X,\mathbb{R})$, we thus deduce that  $\vol_X(\alpha+t\beta)=Z(\alpha+t\beta)^n$ is an a.e. differentiable function with respect to $t$. Therefore, if $\vol_X(\alpha+t\beta)$ and $a_i(t)$ are both differentiable at $t=t_0$, we have
$$
\left.\frac{d}{dt}\right|_{t=t_0}\vol_X(\alpha+t\beta)=nZ(\alpha+t_0\beta)^{n-1}\cdot (\beta-\sum_{i=1}^{r}{a_i}'(t_0)N_i)=nZ(\alpha+t_0\beta)^{n-1}\cdot \beta.
$$
Since $\vol_X(\alpha+t\beta)$ is increasing and continuous, it is also a.e. differentiable and thus we have
\begin{eqnarray}
\label{intergral}
\vol_X(\alpha+s\beta)&=& \vol_X(\alpha)+\int_{0}^{s}\frac{d}{dt}\vol_X(\alpha+t\beta) \mathrm{d}t\nonumber\\
&=&\vol_X(\alpha)+n\int_{0}^{s} Z(\alpha+t\beta)^{n-1}\cdot \beta\ \mathrm{d}t.
\end{eqnarray}
for $0\leq s\leq \epsilon$. Since $Z(\alpha+t\beta)$ is continuous (by Theorem \ref{Continuous}), by $(\ref{intergral})$ we deduce that $\vol_X(\alpha+t\beta)$ is differentiable with respect to $t$ and its derivative
$$
\left.\frac{d}{dt}\right|_{t=t_0}\vol_X(\alpha+t\beta)=nZ(\alpha+t_0\beta)^{n-1}\cdot \beta.\qedhere
$$
\end{proof}

In order to prove transcendental Morse inequality, we will need the following bigness criterion obtained in \cite{Xia13} and \cite{Popo14}.

\begin{thm}\it
Let $X$ be an $n$-dimensional compact K\"ahler manifold. Assume $\alpha$ and $\beta$ are two nef classes on $X$ satisfying
$\alpha^n-n\alpha^{n-1}\cdot\beta>0$,
then $\alpha-\beta$ is a big class.
\end{thm}

The proof of the next theorem is similar to that of Theorem \ref{morse 2} and is therefore omitted.
\begin{thm}\it
\label{morse special}
Let $X$ be a compact K\"ahler manifold on which the modified nef cone $\mathcal{MN}$ and the nef cone $\mathcal{N}$ coincide. If $\alpha$ and $\beta$ are nef cohomology classes of type (1,1) on $X$ satisfying the inequality ${\alpha}^n- n\alpha^{n-1} \cdot \beta>0$. Then $\alpha-\beta$ contains a K\"ahler current and $\vol_X(\alpha-\beta) \geq {\alpha}^{n-1}- n\alpha^{n-1} \cdot \beta$.
\end{thm}

\begin{rem}
In \cite{BCJ09}, the authors proved the following differentiability theorem:
\begin{equation}
\label{bou diff}
\left.\frac{d}{dt}\right|_{t=t_0}\vol_X(L+tD)=n\langle L^{n-1}\rangle \cdot D,
\end{equation}
where $L$ is a big line bundle on the  smooth projective variety $X$ and $D$ is a prime divisor. The right-hand side of the equation above involves the \textit{positive intersection product} $\langle L^{n-1}\rangle\in H^{n-1,n-1}_{\geq0}(X,\mathbb{R})$, first introduced in the analytic context in \cite{BDPP13}. Theorem \ref{morse high} could be seen as a transcendental version of (\ref{bou diff}) for some special K\"ahler manifolds. In the general K\"ahler situation, we propose the following conjecture:
\begin{conj}
Let $X$ be a K\"ahler manifold of dimensional $n$, $\alpha$ be a big class. If $\beta$ is a pseudo-effective class, then we have
$$
\left.\frac{d}{dt}\right|_{t=0}\vol_X(\alpha+t\beta)=n\langle \alpha^{n-1}\rangle \cdot \beta.
$$
\end{conj}
\end{rem} 

\section{Generalized Okounkov bodies on K\"ahler manifolds}
\subsection{Definition and relation with the algebraic case}
\label{defin}
Throughout this subsection, $X$ will stand for a K\"ahler manifold of dimensional $n$. Our main goal in this subsection is to generalize the definition of Okounkov body to any pseudo-effective class $\alpha\in H^{1,1}(X,\mathbb{R})$. First of all, we define a valuation-like function. For any positive current $T\in \alpha$ with analytic singularites, we define the valuation-like function
$$  T \rightarrow \nu(T)=\nu_{Y_\bullet}(T)=(\nu_1(T),\ldots\nu_n(T))
$$
as follows. First,  set
$$\nu_1(T)=\sup\{\lambda \mid T-\lambda [Y_1] \geq 0 \},
$$
where $[Y_1]$ is the current of integration over $Y_1$. By Section \ref{Siu} we know that $\nu_1(T)$ is the coefficient $\nu(T,Y_1)$ of the positive current $[Y_1]$ appearing in the Siu decomposition of $T$. Since $T$ has analytic singularities, by the arguments in Section \ref{analytic singularity}, $T_1:=(T-\nu_1[Y_1])|_{Y_1}$ is a well-defined positive current in the pseudo-effective class $(\alpha-\nu_1\{Y_1\})|_{Y_1}$ and it also has analytic singularities. Then take
$$\nu_2(T)=\sup\{\lambda \mid T_1-\lambda [Y_2] \geq 0 \},
$$
and continue in this manner to define the remaining values $\nu_i(T)\in\mathbb{R}^+$.

\begin{rem}
\label{section divisor}
If one assumes $\alpha\in {\rm NS}_\mathbb{Z}(X)$, there exists a holomorphic line bundle $L$ such that $\alpha=c_1(L)$. If $D$ is the divisor of some holomorphic section $s_D\in H^0(X,\oc_X(L))$, then we have
$$
\nu([D])=\mu(s_D),
$$
where $\mu$ is the valuation-like function appeared in the definition of the original Okounkov body. Roughly speaking our definition of valuation-like function has a bigger domain of definition and thus the image of our valuation-like function contains $\bigcup_{m=1}^{\infty} \frac{1}{m}\mu(mL)$.
\end{rem}

For any big class $\alpha$, we define a \emph{$\mathbb{Q}$-convex body} $\Delta_\mathbb{Q}(\alpha)$ (resp. \emph{$\mathbb{R}$-convex body} $\Delta_\mathbb{R}(\alpha)$) to be the set of valuation vectors $\nu(T)$, where $T$ ranges among all the K\"ahler (resp. positive) currents with algebraic (resp. analytic) singularities. Then $\Delta_\mathbb{Q}(\alpha)\subseteq \Delta_\mathbb{R}(\alpha)$.
It is easy to check that this is a convex set in $\mathbb{Q}^n$ (resp. $\mathbb{R}^n$). Indeed, for any two positive currents $T_0$ and $T_1$ with algebraic (resp. analytic) singularities, we have $\nu(\epsilon T_0+(1-\epsilon)T_1)=\epsilon\nu(T_0)+(1-\epsilon)\nu(T_1)$ for $0\leq \epsilon \leq 1$ rational (resp. real). It is also obvious to see the homogeneous property of  $\Delta_\mathbb{Q}(\alpha)$, that is, for all $c\in \mathbb{Q}^+$, we have $$\Delta_\mathbb{Q}(c\alpha)=c\Delta_\mathbb{Q}(\alpha).$$ Indeed, since we have $\nu(cT)=c\nu(T)$ for all $c\in \mathbb{R}^+$, the claim follows directly.

\begin{example}
\label{curve}
Let $L$ be a line bundle of degree $c>0$ on a smooth curve $C$ of genus $g$. Then we have
$$
\Delta_\mathbb{Q}(c_1(L))=\mathbb{Q}\cap [0,c).
$$
Since ${\rm NS}_\mathbb{R}(C)=H^{1,1}(C,\mathbb{R})$, for any ample class $\alpha$ on $C$ we have
$$
\Delta_\mathbb{Q}(\alpha)=\mathbb{Q}\cap [0,\alpha\cdot C).
$$
\end{example}

\begin{lem}\it
\label{boundedness}
Let $\alpha$ be a big class, then the $\mathbb{R}$-convex body $\Delta_\mathbb{R}(\alpha)$ lies in a bounded subset of $\mathbb{R}^n$.
\end{lem}
\begin{proof}
It suffices to show that there exists a $b>0$ large enough such that $\nu_i(T)<b$ for any positive current $T$ with analytic singularities. We fix a K\"ahler class $\omega$. Choose first of all $b_1>0$ such that $$
(\alpha-b_1Y_1)\cdot \omega^{n-1}<0.
$$
This guarantees that $\nu_1(T)<b_1$ since $\alpha-b_1Y_1\not\in \mathcal{E}$. Next choose $b_2$ large enough so that
$$
((\alpha-aY_1)|_{Y_1}-b_2Y_2)\cdot \omega^{n-2}<0
$$
for all real numbers $0\leq a\leq b_1$. Then $\nu_2(T)\leq b_2$ for any positive current $T$ with analytic singularities. Continuing in this manner we construct $b_i>0$ for $i=1,\ldots,n$ such that $\nu_i(T)\leq b_i$ for any positive current $T$ with analytic singularities. We take $b=\max\{b_i\}$.
\end{proof}

\begin{lem}\it
\label{closure same}
For any big class $\alpha$, $\Delta_\mathbb{Q}(\alpha)$ is dense in $\Delta_\mathbb{R}(\alpha)$. Thus we have $\overline{\Delta_\mathbb{Q}(\alpha)} = \overline{\Delta_\mathbb{R}(\alpha)}$.
\end{lem}

\begin{proof}
It is easy to verify that if $T$ is a K\"ahler current with analytic singularities, then for any $\epsilon>0$, there exists  a K\"ahler current $S_\epsilon$ with algebraic singularities such that $\left\lVert\nu(S_\epsilon)-\nu(T)\right\rVert<\epsilon$ with respect to the standard norm in $\mathbb{R}^n$. For the general case, We fix a K\"ahler current $T_0\in i\Theta(L)$ with algebraic singularities. Then for any positive current $T$ with analytic singularities, $T_{\epsilon}:=(1-\epsilon )T+\epsilon T_0$ is still a K\"ahler current. By Lemma \ref{boundedness}, $\left\lVert\nu(T_{\epsilon})-\nu(T)\right\rVert = \epsilon\left\lVert(\nu(T_0)-\nu(T))\right\rVert$ will tend to 0 since $\nu(T)$ is uniformly bounded for any positive current $T$ with analytic singularities. Thus $\Delta_\mathbb{Q}(\alpha)$ is dense in $\Delta_\mathbb{R}(\alpha)$.
\end{proof}

Now we study the relations between $\Delta_{\mathbb{Q}}(c_1(L))$ and $\Delta(L)$ for $L$ a big line bundle on $X$. First we begin with the following two lemmas.

\begin{lem}[Extension property]\it
\label{Extension property}
Let $L$ be a big line bundle on the projective variety $X$ of dimension $n$, with a singular Hermitian metric $h=e^{-\varphi}$ satisfying
$$
i\Theta_{L,h}=dd^c\varphi\geq \epsilon\omega
$$
for some $\epsilon>0$ and a given K\"ahler form $\omega$. If the restriction of $\varphi$ on a smooth hypersurface $Y$ is not identically equal to $-\infty$, then there exists a positive integer $m_0$ which depends only on $Y$ so that any holomorphic section $s_m\in H^0(Y,\oc_Y(mL)\otimes \mathcal{I}(m\varphi|_Y))$ can be extended to $ S_m\in H^0(X,\oc_X(mL)\otimes \mathcal{I}(m\varphi))$ for any $m\geq m_0$.
\end{lem}

We need the following Ohsawa-Takegoshi extension theorem to prove Lemma \ref{Extension property}.
\begin{thm}[Ohsawa-Takegoshi]\it
\label{Ohsawa}
Let $X$ be a smooth projective variety. Let $Y$ be a smooth divisor defined by a holomorphic section of the line bunle $H$ with a smooth metric $h_0=e^{-\psi}$. Let L be a holomorphic line bunle with a singular metric $h=e^{-\phi}$, satisfying the curvature assumptions
$$
dd^c\phi\geq0
$$
and
$$
dd^c\phi\geq \delta dd^c\psi
$$
with $\delta>0$. Then for any holomorphic section $s\in H^0(Y,\oc_Y(K_Y+L)\otimes \mathcal{I}(h|_Y))$, there exists a global holomorphic section $S\in H^0(X,\oc_X(K_X+L+Y)\otimes \mathcal{I}(h))$ such that $S|_Y=s$.
\end{thm}

\begin{proof}[Proof of Lemma \ref{Extension property}]
Taking a smooth metric $e^{-\psi}$ and $e^{-\eta}$ on $Y$ and $K_X$, we can choose $m_0$ large enough satisfying the curvature assumptions
$$
dd^c(m\phi-\eta-\psi)\geq0
$$
and
$$
dd^c(m\phi-\eta-\psi)\geq dd^c\psi
$$
for any $m\geq m_0$.

By Theorem \ref{Ohsawa}, any holomorphic section $s\in H^0(Y,\oc_Y(K_Y+(mL-K_X-Y)|_Y)\otimes \mathcal{I}(h^m|_Y))$ can be extended to a global holomorphic section $S\in H^0(X,\oc_X(mL)\otimes \mathcal{I}(h^m))$ such that $S|_Y=s$. By the adjunction theorem, we have $(K_X+Y)|_Y=K_Y$, thus the lemma is proved.

\end{proof}

\begin{lem}\it
\label{Riemann}
Let $L$ be a big line bundle on the Riemann surface $C$ with a singular Hermitian metric $h=e^{-\varphi}$ such that $\varphi$ has algebraic singularities and
$$
i\Theta_{L,h}=dd^c\varphi\geq \epsilon\omega
$$
for some $\epsilon>0$. Then for a fixed point $p$, there exists an integer $k>0$ such that we have a holomorphic section $s_k\in H^0(C,\oc_C(kL)\otimes \mathcal{I}(h^k))$ satisfying $\ord_p(s_k)=k\nu(i\Theta_{L,h},p)$.
\end{lem}
\begin{proof}
Since $\varphi$ has algebraic singularities, we have the following Lebsegue decomposition
$$
i\Theta_{L,h}=({i\Theta_{L,h}})_{\rm ac}+\sum_{i=1}^{r}c_ix_i,
$$
where each $c_i>0$ is rational and $x_1,\ldots,x_r$ are the log poles of $i\Theta_{L,h}$ (possibly $p$ is among them).
Since we have $$\int_Ci(\Theta_{L,h})_{\rm ac}+\sum_{i=1}^{r}c_i=\deg(L),$$
thus
$$
\sum_{i=1}^{r}c_i<\deg(L).
$$

By Riemann-Roch theorem there exists an integer $k>0$ satisfying\\
\begin{enumerate}[\upshape (i)]
\item $kc_i$ is integer,
\item there is a holomorphic section $s_k\in H^0(C,\oc_C(kL))$ such that ${\rm ord}_{x_i}(s_k)\geq kc_i$ and ${\rm ord}_p(s_k)=k\nu(i\Theta_{L,h},p)$.
\end{enumerate}
Thus $s_k$ is locally integrable with respect to the weight $e^{-k\varphi}$. The theorem is proved.

\end{proof}

\begin{thm}\it
\label{approximation}
Let $X$ be a smooth projective manifold of dimension $n$. For any K\"ahler current $T\in c_1(L)$ with algebraic singularities, there exists a holomorphic section $s\in H^0(X,\oc_X(kL))$ such that $\mu(s)=k\nu(T)$, i.e.,  we have $$\nu(T)\in \displaystyle\bigcup_{m=1}^{\infty} \frac{1}{m}\mu(mL).$$
In particular,
$$
\Delta_{\mathbb{Q}}(c_1(L))\subseteq \displaystyle\bigcup_{m=1}^{\infty} \frac{1}{m}\mu(mL) \subseteq \Delta(L).
$$
\end{thm}
\begin{proof}
First, set $\nu_i=\nu_i(T)$ and define
$$ T_0:=T,\ T_1:=(T_0-\nu_1[Y_1])|_{Y_1},\ \ldots\ , T_{n-1}:=(T_{n-2}-\nu_{n-1}[Y_{n-1}])|_{Y_{n-1}};
$$
$$L_0:=L-\nu_1Y_1, \ L_1:=L_0|_{Y_1}-\nu_2Y_2,\ \ldots\ ,
L_{n-2}:=L_{n-3}|_{Y_{n-2}}-\nu_{n-1}Y_{n-1}.
$$

Since $T_0\geq \epsilon \omega$, we have $T_1\geq \epsilon \omega|_{Y_1}$, \ldots , $T_{n-1}\geq \epsilon \omega|_{Y_{n-1}}$. Since each $\nu_i$ is rational, we could find an integer $m$ to make each $m\nu_i$ be integer so that each $mL_i$ is a big line bundle on $Y_i$. If we could prove
$$\nu(mT)\in \displaystyle\bigcup_{k=1}^{\infty} \frac{1}{k}\mu(kmL),$$
then we will have
$$\nu(T)\in \displaystyle\bigcup_{m=1}^{\infty} \frac{1}{m}\mu(mL),$$
by the homogeneous property $\frac{1}{m}\nu(mT)=\nu(T)$.
Thus we can assume that each $\nu_i(T)$ is an integer after we replace $L$ by $mL$ and $T$ by $mT$.

Firstly, since $T_0\in c_1(L)$ is a K\"ahler current with algebraic singularities, there exists a singular metric $h=e^{-\varphi_0}$ on $L$ whose curvature current is $T_0$ and $\varphi$ has algebraic singularities; on the other hand,  there is a canonical metric $e^{-\eta_0}$ on $\oc_{Y_0}(Y_1)$ such that $dd^c\eta_0=[Y_1]$ in the sense of currents, thus by the definition of $\nu_1$ we deduce that $h_0:=e^{-\varphi_0+\nu_1\eta_0}$ is a singular metric of $L_0$ such that $-\varphi_0+\nu_1\eta_0$ does not vanish identically on $Y_1$, and $h_0|_{Y_1}$ is a singular metric of $L_0|_{Y_1}$  with algebraic singularities whose curvature current is $T_1\geq \epsilon\omega|_{Y_1}$.

Secondly, there is a canonical singular metric $e^{-\eta_1}$ of $\oc_{Y_1}(Y_2)$ on $Y_1$
with the curvature current $[Y_2]$. Thus the singular metric $h_1:=h_0|_{Y_1}+e^{\nu_2\eta_1}$ of the big line bundle $L_1$ gives a curvature current $T_1-\nu_2[Y_2]\geq \epsilon\omega|_{Y_1}$. We continue in this manner to define the remaining singular metrics $h_i:=h_{i-1}|_{Y_i}+e^{\nu_{i+1}\eta_i}$ of the big line bundle $L_i$  on $Y_i$ with curvature current $T_i-\nu_{i+1}[Y_{i+1}]\geq \epsilon\omega|_{Y_i}$ for $i=0, \ldots ,n-1$. It is easy to see that $h_i|_{Y_{i+1}}$ is well-defined.

By Lemma \ref{Extension property}, there exists a $k_0$ such that for each $k\geq k_0$, the following short sequence is exact
\begin{equation}
\label{exact}
H^0(Y_{i-1},\oc_{Y_{i-1}}(kL_{i-1})\otimes \mathcal{I}(h^k_{i-1}))\longrightarrow H^0(Y_{i},\oc_{Y_{i}}(kL_{i-1})\otimes \mathcal{I}(h^k_{i-1}|_{Y_{i}}))\longrightarrow 0
\end{equation}
for $i=1, \ldots ,n-1$.

Now we begin our construction. $T_{n-1}$ is the curvature current of the singular metric  $h_{n-2}|_{Y_{n-1}}$ of $L_{n-2}|_{Y_{n-1}}$ over the Riemann surface $Y_{n-1}$. By Lemma \ref{Riemann}, there exists a $k\geq k_0$ and a holomorphic section $s_{n-1} \in H^0(Y_{n-1},\oc_{Y_{n-1}}(kL_{n-2})\otimes \mathcal{I}(h^k_{n-2}|_{Y_{n-1}}))$, such that $\ord_p(s_{n-1})=k\nu(T_{n-1},p)=k\nu_n$.

By the exact sequence (\ref{exact}), $s_{n-1}$ could be extend to
$$
\widetilde{s}_{n-2}\in H^0(Y_{n-2},\oc_{Y_{n-2}}(kL_{n-2})\otimes \mathcal{I}(h^k_{n-2})).
$$
Now we choose a canonical section $t_{n-2}$ of  $H^0(Y_{n-2},\oc_{Y_{n-2}}(Y_{n-1}))$ such that the divisor of $t_{n-2}$ is $Y_{n-1}$. We define $s_{n-2}:=\widetilde{s}_{n-2}{t^{\otimes \nu_{n-1}}_{n-2}}$, by the construction of $h_{n-2}:=h_{n-3}|_{Y_{n-2}}+e^{\nu_{n-1}\eta_{n-2}}$, we obtain that
$$
s_{n-2}\in H^0(Y_{n-2},\oc_{Y_{n-2}}(kL_{n-3})\otimes \mathcal{I}(h^k_{n-3}|_{Y_{n-2}}).
$$
We can continue in this manner to construct a section $s_{0}\in H^0(X,\oc_{X}(kL))$ and by our construction we have  $$
\mu(s_0) =(k\nu_1, \ldots ,k\nu_n)=k\nu(T),
$$
this concludes the theorem.

\end{proof}

\begin{proposition}\it
\label{coincide}
For any big line bundle $L$ and any admissible flag $Y_{\bullet}$, one has $\overline{\Delta_{\mathbb{Q}}(c_1(L))}=\Delta(L)$. In particular,
$$\Delta(L)= \overline{\displaystyle\bigcup_{m=1}^{\infty} \frac{1}{m}\nu(mL)}.$$
\end{proposition}

\begin{proof}
Firstly, since $\Delta_{\mathbb{Q}}(c_1(L))$ is a convex set in $Q^n$, its closure $\overline{\Delta_{\mathbb{Q}}(c_1(L))} $ is also a closed convex set in $\mathbb{R}^n$. By Proposition \ref{approximation}, we have
$$
\Delta_{\mathbb{Q}}(c_1(L)) \subset \displaystyle\bigcup_{m=1}^{\infty} \frac{1}{m}\cdot \nu(mL),
$$
thus
$$
\overline{\Delta_{\mathbb{Q}}(c_1(L))}\subseteq \Delta(L).
$$
By Remark \ref{section divisor}, we have $\bigcup_{m=1}^{\infty} \frac{1}{m}\nu(mL) \subseteq \Delta_{\mathbb{R}}(c_1(L))$,
thus by the definition of Okounkov body $\Delta(L)$, we deduce that
$$
\Delta(L)\subseteq \overline{\Delta_{\mathbb{R}}(c_1(L))}.
$$
By Lemma \ref{closure same}, we have $\overline{\Delta_\mathbb{Q}(c_1(L))} = \overline{\Delta_\mathbb{R}(c_1(L))}$, thus the theorem is proved.

\end{proof}

\begin{rem}
By Proposition \ref{coincide}, in the definition of the Okounkov body $\Delta(L)$, it suffices to close up the set of normalized valuation vectors instead of the closure of the convex hull of this set.
\end{rem}

\begin{rem}
It is easy to reprove that the Okounkov body $\Delta(L)$ depends only on the numerical equivalence class of the big line bundle $L$. Indeed, if $L_1$ and $L_2$ are numerically equivalent, we have $c_1(L_1)=c_1(L_2)$ thus
$$
\Delta_{\mathbb{Q}}(c_1(L_1))=\Delta_{\mathbb{Q}}(c_1(L_2)).
$$
By Proposition \ref{coincide}, we have
$$
\Delta(L_1)=\Delta(L_2).
$$
\end{rem}

Now we are ready to find some valuative points in the Okounkov bodies.
\begin{proof}[Proof of Corollary \ref{valuation}]
In \cite{LM09} we know that $\vol_{\mathbb{R}^n}(\Delta(L))=\vol_X(L)>0$ by the bigness of $L$. Since we have $\Delta(L)=\overline{\Delta_{\mathbb{Q}}(c_1(L))}$ by Proposition \ref{coincide}, then for any  $p\in {\rm int}(\Delta(L))\cap\mathbb{Q}^n$, there exists an $n$-simplex $\Delta_n$ containing $p$ with all the vertices lying in $\Delta_{\mathbb{Q}}(c_1(L))$.  Since $\Delta_{\mathbb{Q}}(c_1(L))$ is a  convex set in $\mathbb{Q}^n$, we have $\Delta_n\cap\mathbb{Q}^n\subseteq \Delta_{\mathbb{Q}}(c_1(L))$, and thus
$$
\Delta_{\mathbb{Q}}(c_1(L))\supseteq {\rm int}(\Delta(L))\cap\mathbb{Q}^n.
$$
From Theorem \ref{approximation} we have $\Delta_{\mathbb{Q}}(c_1(L))\subseteq \displaystyle\bigcup_{m=1}^{\infty} \frac{1}{m}\mu(mL)$, thus we get the inclusion
$$
{\rm int}(\Delta(L))\cap\mathbb{Q}^n\subseteq \displaystyle\bigcup_{m=1}^{\infty} \frac{1}{m}\mu(mL),
$$
which means that all rational interior points of $\Delta(L)$ are valuative.
\end{proof}

Pursuing the same philosophy as in Proposition \ref{coincide}, it is natual to extend results related to Okounkov bodies for big line bundles, to the more general case of an arbitrary big class $\alpha\in H^{1,1}(X,\mathbb{R})$. We propose the following definition.

\begin{defin}[Generalized Okounkov body]
\label{generalized Okounkov}
Let $X$ be a K\"ahler manifold of dimension $n$. We define the \emph{generalized Okounkov body} of a big class $\alpha\in H^{1,1}(X,\mathbb{R})$ with respect to the fixed flag $Y_\bullet$ by
$$
\Delta(\alpha)= \overline{\Delta_{\mathbb{R}}(\alpha)}=\overline{\Delta_{\mathbb{Q}}(\alpha)}.
$$
\end{defin}

We have the following properties for generalized Okounkov bodies:
\begin{proposition}\it
\label{body continuity}
Let $\alpha$ and $\beta$ be big classes, $\omega$ be any K\"ahler class. Then:
\begin{enumerate}[\upshape (i)]
\item $
\Delta(\alpha)+\Delta(\beta)\subseteq \Delta(\alpha+\beta).
$
\item $\vol_{\mathbb{R}^n}(\Delta(\omega))>0.$
\item $
\Delta(\alpha)=\bigcap_{\epsilon>0}\Delta(\alpha+\epsilon\omega).
$
\end{enumerate}
\end{proposition}

\begin{proof}
(i) is obvious from the definition of generalized Okounkov body. To prove (ii), we use induction for dimension. The result is obvious if $n=1$, assume now that (ii) is true for $n-1$. We choose $t>0$ small enough such that $\omega-tY_1$ is still a K\"ahler class. By the main theorem of \cite{CT14}, any K\"ahler current $T\in (\omega-tY_1)|_{Y_1}$ with analytic singularities can be extended to a K\"ahler current $\widetilde{T}\in \omega-tY_1$, thus we have 
$$
\Delta(\omega)\bigcap t\times{\mathbb{R}}^{n-1}=t\times \Delta((\omega-tY_1)|_{Y_1}),
$$
where $\Delta((\omega-tY_1)|_{Y_1})$ is the generalized Okounkov body of $(\omega-tY_1)|_{Y_1}$ with respect to the flag
$$
Y_1\supset Y_2\supset \ldots \supset Y_n=\{p\}.
$$
By the induction, we have $\vol_{\mathbb{R}^{n-1}}(\Delta((\omega-tY_1)|_{Y_1}))>0$. Since $\Delta(\omega)$ contains the origin, we have $\vol_{\mathbb{R}^n}(\Delta(\omega))>0.$

Now we are ready to prove (iii). By the concavity we have
$$\Delta(\alpha+\epsilon_1\omega)+\Delta((\epsilon_2-\epsilon_1)\omega)\subseteq \Delta(\alpha+\epsilon_2\omega)$$ 
if $0\leq \epsilon_1< \epsilon_2$. Since $\Delta(\omega)$ contains the origin, we have 
$$
\Delta(\alpha)\subseteq\bigcap_{\epsilon>0}\Delta(\alpha+\epsilon\omega),
$$
and 
$$
\Delta(\alpha+\epsilon_1\omega)\subseteq \Delta(\alpha+\epsilon_2\omega).
$$

From the concavity property, we conclude that  $\vol_{\mathbb{R}^n}(\Delta(\alpha+t\omega))$ is a concave function for $t\geq 0$, thus continuous. Then we have 
$$
\vol_{\mathbb{R}^n}(\bigcap_{\epsilon>0}\Delta(\alpha+\epsilon\omega))=\vol_{\mathbb{R}^n}(\Delta(\alpha))>0.
$$
Since they are all closed and convex, we have 
$$
\Delta(\alpha)=\bigcap_{\epsilon>0}\Delta(\alpha+\epsilon\omega).
$$
\end{proof}

\begin{rem}
We don't know whether $\vol_{\mathbb{R}^n}(\Delta(\alpha))$ is independent of  the choice of the admissible flag. However, in the next subsection we will prove that in the case of surfaces we have 
$$
\vol_X(\alpha)=2\vol_{\mathbb{R}^2}(\Delta((\alpha)),
$$
in particular the Euclidean volume of the generalized Okounkov body is independent of the choice of the flag. We conjecture that 
$$
\vol_{\mathbb{R}^n}(\Delta(\alpha))=\frac{1}{n!}\cdot \vol_X(\alpha),
$$
as we proposed in the introduction.
\end{rem}

\subsection{Generalized Okounkov bodies on complex surfaces}
Now we will mainly focus on generalized Okounkov bodies of compact K\"ahler surfaces. In this section, $X$ denotes a compact K\"ahler surface. We fix henceforth an admissible flag
$$
X\supseteq C\supseteq \{x\},
$$
on $X$, where $C\subset X$ is an irreducible curve and $x\in C$ is a smooth point.

\

\begin{defin}
\label{restrict dfn}
For any big class $\alpha\in H^{1,1}(X,\mathbb{R})$, we denote the \emph{restricted $\mathbb{R}$-convex body} of $\alpha$ along $C$ by $\Delta_{\mathbb{R},X|C}(\alpha)$, which is defined to be the set of Lelong numbers $\nu(T|_C,x)$, where $T\in \alpha$ ranges among all the positive currents with analytic singularities such that $C\not\subseteq E_+(T)$.
The \emph{restricted Okounkov body} of $\alpha$ along $C$ is defined as
$$\Delta_{X\mid C}(\alpha):=\overline{\Delta_{\mathbb{R},X|C}(\alpha)}.$$
\end{defin}

When $\alpha=c_1(L)$ for some big line bunle $L$ on X, it is noticeable that $\Delta_{X|C}(\alpha)=\Delta_{X|C}(L)$, where $\Delta_{X|C}(L)$ is defined in \cite{LM09}. When $L$ is ample, we have $\Delta_{X|C}(L)=\Delta(L|_C)$. Indeed, it is suffice to show that for any section $s\in H^0(C,\oc_C(L))$, there exists an integer $m$ such that $s^{\otimes m}$ can be extended to a section $S_m\in H^0(X,\oc_X(mL))$.  This can be garanteed by Kodaira vanishing theorem. When $\alpha$ is any ample class, there is a very similar theorem which has appeared in the proof of Proposition \ref{body continuity}. However, the proof there relies on the difficult extension theorem in \cite{CT14}. Here we give a simple and direct proof when $X$ is a complex surface. Anyway, the idea of proof here is borrowed from \cite{CT14}.

\begin{proposition}\it
\label{ample}
If $\alpha$ is an ample class, then we have
$$
\Delta_{X| C}(\alpha)=\Delta(\alpha|_C)=[0,\alpha\cdot C].
$$
\end{proposition}

\begin{proof}
From Definition \ref{restrict dfn}, we have $\Delta_{X| C}(\alpha)\subseteq \Delta(\alpha|_C)$. It suffices to prove that for any K\"ahler current $T\in \alpha|_C$ with mild analytic singularities, we have a positive current $\widetilde{T}\in \alpha$ with analytic singularites such that $\widetilde{T}|_C=T$. First we choose a K\"ahler form $\omega\in\alpha$. By assumption, we can write $T=\omega|_V + dd^c\varphi$ for some quasi-plurisubharmonic function $\varphi$ on $C$ which has mild analytic singularities. Our goal is to extend  $\varphi$ to a function $\Phi$ on $X$ such that $\omega + dd^c\Phi$ is a K\"ahler current with analytic singularities.

Choose $\epsilon>0$ small enough so that $$T=\omega|_C + idd^c\varphi\geq3\epsilon\omega,$$
holds as currents on $C$. We can cover $C$ by finitely many charts $ \{W_j\}_{1\leq j\leq N}$ satisfying the following properties:
\begin{enumerate}[\upshape (i)]
\item On each $W_j(j\leq k)$ there are local coordinates $(z^{(j)}_1,z^{(j)}_2)$ such that $C\bigcap W_j=\{z^{(j)}_2=0\}$ and $$\varphi=\frac{c_j}{2}\log |z^{(j)}_1|^2+g_j(z^{(j)}_1)$$
where $g_j(z^{(j)}_1)$ is smooth and bounded on $W_j\bigcap C$. We denote the single pole of $T$ in $W_j(j\leq k)$ by $x_j$;
\item On each $W_j(j>k)$ the local potential $\varphi$ is smooth and bounded on $W_j\bigcap C$;
\item $x_i\not\in\overline{W_j}$ for $i=1,\ldots,k$ and $j\neq i$.
\end{enumerate}
Define a function $\varphi_j$ on $W_j$ (with analytic singularities) by
\begin{equation}
\varphi_j(z^{(j)}_1,z^{(j)}_2)=\left\{
\begin{aligned}
&\varphi(z^{(j)}_1)+A|z^{(j)}_2|^2 \quad &\text{if} \quad j>k,\\
&\frac{c_j}{2}\log (|z^{(j)}_1|^2+|z^{(j)}_2|^2)+g_j(z^{(j)}_1)+A|z^{(j)}_2|^2 \quad &\text{if} \quad j\leq k,
\end{aligned}
\right.\nonumber
\end{equation}
where $A>0$ is a constant.  If we shrink the $W_j$'s slightly, still preserving the property that $C\subseteq \bigcup W_j$, we can choose $A$ sufficiently large so that
$$
\omega+dd^c\varphi_j\geq2\epsilon\omega,
$$
holds on $W_j$ for all $j$. We also need to construct slightly smaller open sets $W'_j\subset \subset U_j \subset \subset W_j $ such that $\bigcup W'_j$ is still a covering of $C$.

By construction $\varphi_j$ is smooth when $j>k$, and $\varphi_j$ is smooth outside the log pole $x_j$ when $j\leq k$. By property (iii) above, we can glue the functions $\varphi_j$ together to produce a K\"ahler current
$$\widetilde{T}=\omega|_U+dd^c\widetilde{\varphi}\geq \epsilon\omega$$
defined in a neighborhood $U$ of $C$ in $X$, thanks to Richberg's gluing procedure. Indeed, $\varphi_i$ is smooth on $W_i\bigcap W_j$ for any $j\neq i$, which is a sufficient condition in using the Richberg technique.   From the construction of $\widetilde{\varphi}$, we know that  $\widetilde{\varphi}|_C=\varphi$, $\widetilde{\varphi}$ has log poles in every $x_i$ and is continuous outside $x_1,\ldots,x_k$.

On the other hand, since $\alpha$ is an ample class, there exists a rational number $\delta>0$ such that $\alpha-\delta\{C\}$ is still ample, thus we have a K\"ahler form $\omega_1\in \alpha-\delta\{C\}$. We can write $\omega_1+\delta[C]=\omega+dd^c\phi$, where $\phi$ is smooth outside $C$, and for any point $x\in C$, we have
$$
\phi=\frac{\delta}{2}\log |z_2|^2+O(1),
$$
where $z_2$ is the local equation of $C$.

Since $\phi$ is continuous outside $C$, we can choose a large constant $B>0$ such that $\phi>\widetilde{\varphi}-B$ in a neighborhood of $\partial U$. Therefore we define

$$\Phi=\begin{cases}
\max\{\widetilde{\varphi},\phi+B\}\ &\text{on}\ U\\
\phi+B &\text{on}\ X-U,
\end{cases}$$
which is well defined on the whole of $X$, and satisfies $\omega+dd^c\Phi\geq \epsilon'\omega$ for some $\epsilon'>0$. Since $\phi=-\infty$ on $C$, while $\widetilde{\varphi}|_C=\varphi$, it follows that $\Phi|_C=\varphi$.

We claim that $\Phi$ also has analytic singularities. Since around $x_j$, we have
$$
\widetilde{\varphi}(z_1,z_2)=\frac{c_j}{2}\log (|z_1|^2+|z_2|^2)+O(1),
$$
and
$$
\phi(z_1,z_2)=\frac{\delta}{2}\log |z_2|^2+O(1),
$$
for some local coordinates $(z_1,z_2)$ of $x_j$. Thus locally we have
$$
\max\{\widetilde{\varphi},\phi+A\}=\frac{1}{2}\log (|z_1|^{2c_j}+|z_2|^{2c_j}+|z_2|^{2\delta})+O(1).
$$
Since $\Phi$ is continuous outside $x_1,\ldots,x_k$, our claim is proved.

\end{proof}

\begin{lem}\it
\label{big nef}
Let $\alpha$ be a big and nef class on $X$, then for any $\epsilon>0$, there exsists a K\"ahler current $T_\epsilon\in \alpha$ with analytic singularities such that the Lelong number $\nu(T_\epsilon,x)<\epsilon$ for any point in $X$. Moreover, $T_\epsilon$ also satisfies
$$
E_+(T)=E_{nK}(\alpha).
$$
\end{lem}
\begin{proof}
Since $\alpha$ is big, there exists a K\"ahler current with analytic singularities such that $E_+(T_0)=E_{nK}(\alpha)$ and $T_0>\omega$ for some K\"ahler form $\omega$. Since $\alpha$ is also a nef class, for any $\delta>0$, there exists a smooth form $\theta_\delta$ such that $\theta_\delta\geq -\delta\omega$. Thus $T_\delta:=\delta T_0+(1-\delta)\theta_\delta\geq \delta^2\omega$ is a K\"ahler current with analytic singularities satisfying that $$E_+(T_\delta)=E_+(T_0)=E_{nK}(\alpha),$$
and
$$\nu(T_\delta,x)=\delta\nu(T_0,x)$$
for $x\in X$. Since the Lelong number $\nu(T_0,x)$ is an upper continuous function (thus bounded from above), $\nu(T_\delta,x)$  converges uniformly to zero as $\delta$ tends to 0. The lemma is proved.
\end{proof}

\begin{proposition}\it
\label{restrict volume}
Let $\alpha$ be a big and nef class, $C\not\subseteq E_{nK}(\alpha)$. Then we have $$
\Delta_{X| C}(\alpha)=\Delta(\alpha|_C)=[0,\alpha\cdot C].
$$
\end{proposition}

\begin{proof}
Asumme $E_{nK}(\alpha)=\bigcup_{i=1}^{r}C_i$, where  each $C_i$ is an irreducible curve. By Lemma \ref{big nef}, for any $\epsilon>0$ there exists a K\"ahler current $T_\epsilon\in \alpha$ with analytic singularities such that  $$E_+(T_\epsilon)=E_{nK}(\alpha)=\text{\rm Null}(\alpha)=\bigcup_{i=1}^{r}C_i$$
and $\nu(T_\epsilon,x)<\epsilon$ for all $x\in X$. Thus the Siu decomposition
$$T_\epsilon=R_\epsilon+\displaystyle\sum_{i=1}^{r}a_{i,\epsilon}C_i$$
satisfies $0\leq a_{i,\epsilon}<\epsilon$, and $R_\epsilon$ is a K\"ahler current whose analytic singularities are isolated points. By Remark \ref{mf=f}, the cohomology class $\{R_\epsilon\}$ is a K\"ahler class and converges to $\alpha$ as $\epsilon\rightarrow 0$. In particular, $|\{R_\epsilon\}\cdot C-\alpha\cdot C|<A\epsilon$, where $A$ is a constant.

By Proposition \ref{ample}, there exists a K\"ahler current $S_\epsilon\in \{R_\epsilon\}$  with analytic singularities such that $C\not\subseteq E_+(S_\epsilon)$ and $-\epsilon<\nu(S_\epsilon|_C,x)-\{R_\epsilon\}\cdot C<0$. Thus  $T'_\epsilon:=S_\epsilon+\sum_{i=1}^{r}a_{i,\epsilon}C_i$ is a K\"ahler current in $\alpha$ with analytic singularities, and $-(1+A)\epsilon<\nu(T'_\epsilon|_C,x)-\alpha\cdot C$. Since $\alpha$ is big and nef, there exists a K\"ahler current $P_\epsilon$ in $\alpha$ with analytic singularities such that $\nu(P_\epsilon|_C,x)<\epsilon$. Therefore, by the definition of $\Delta_{X| C}(\alpha)$ and the convexity property we deduce that $[0,\alpha\cdot C]\subseteq \Delta_{X| C}(\alpha)$. On the other hand, $\Delta_{X| C}(\alpha)\subseteq \Delta(\alpha|_C)=[0,\alpha\cdot C]$ by definition. The proposition is proved.

\end{proof}

\begin{lem}\it
\label{restrict body}
Let $\alpha$ be a big class on $X$ with divisorial Zariski decomposition $\alpha=Z(\alpha)+N(\alpha)$. Assume
that $C\not\subseteq E_{nK}(Z(\alpha))$, so that $C\not\subseteq {\rm Supp}(N(\alpha))$ by Theorem \ref{contain exceptional}. Moreover, set
$$f(\alpha)=\nu_x(N(\alpha)|_C),\ \ \  g(\alpha)=\nu_x(N(\alpha)|_C)+Z(\alpha)\cdot C,$$
where $\nu_x(N(\alpha)|_C)=\nu(N(\alpha)|_C,x)$.
Then the restricted Okounkov body of $\alpha$ along $C$ is the interval
$$
\Delta_{X|C}(\alpha)=[f(\alpha),g(\alpha)]
$$

\end{lem}

\begin{proof}
First, by Remark \ref{bijection} we conclude that $T\mapsto T-N(\alpha)$ is a bijection between the positive currents in $\alpha$ and those in $Z(\alpha)$, thus we have
$$
E_{nK}(\alpha)=E_{nK}(Z(\alpha))\bigcup \text{supp}(N(\alpha)),
$$
and
\begin{equation}
\label{same nK}
C\not\subseteq E_{nK}(Z(\alpha)) \iff C\not\subseteq E_{nK}(\alpha).
\end{equation}
By the assumption of theorem, $N(\alpha)|_C$ is a well-defined positive current with analytic singularites on $C$.  By the definition of $\Delta_{\mathbb{R},X|C}(\alpha)$, we have
$$
\Delta_{\mathbb{R},X|C}(\alpha)=\Delta_{\mathbb{R},X|C}(Z(\alpha))+\nu_x(N(\alpha)|_C).
$$
We take the closure of the sets to get
$$
\Delta_{X|C}(\alpha)=\Delta_{X|C}(Z(\alpha))+\nu_x(N(\alpha)|_C).
$$
Since $\alpha$ is big, thus $Z(\alpha)$ is big and nef, and by Proposition \ref{restrict volume} we have $\Delta_{X|C}(Z(\alpha))=[0,Z(\alpha)\cdot C]$. We have proved the lemma.

\end{proof}

\begin{defin}
If $\alpha$ is big and $\beta$ is pseudo-effective, then the slope of $\beta$ with respect to $\alpha$ is defined as
$$
s=s(\alpha,\beta)=\sup\{t>0\mid \alpha-t\beta\ \text{is big}\}.
$$
\end{defin}

\begin{rem}
\label{boundary}
Since the big cone is open, we know that $\{t>0\mid \alpha > t\beta\}$ is an open set in $\mathbb{R}^+$. Thus $\alpha-s\beta$ belongs to the boundary of the big cone $\mathcal{E}$, and $\vol_X(\alpha-s\beta)=0$.
\end{rem}

\begin{proof}[Proof of Theorem \ref{Okounkov}]
For $t\in [0,s)$, we put $\alpha_t=\alpha-t\{C\}$, and let $Z_t:=Z(\alpha_t)$ and $N_t:=N(\alpha_t)$ be the positive and negative part of the divisorial Zariski decomposition of $\alpha_t$.

(i) First we assume $C$ is nef. By Theorem \ref{contain exceptional},  the prime divisors in $E_{nK}(Z(\alpha_t))$ form an exceptional family, thus $C\not\subseteq  E_{nK}(Z(\alpha_t))$, thus $C\not\subseteq  E_{nK}(\alpha_t)$ by (\ref{same nK}). By Lemma \ref{restrict body} we have $
\Delta_{X| C}(\alpha_t)=[\nu_x(N_t|_C),Z_t\cdot C+\nu_x(N_t|_C)].
$

By the definition of $\mathbb{R}$-convex body and restrict $\mathbb{R}$-convex body, we have
$$
\Delta_{\mathbb{R}}(\alpha)\bigcap t\times \mathbb{R}= t\times \Delta_{\mathbb{R},X|C}(\alpha_t).
$$
Thus
$$
t\times \overline{\Delta_{\mathbb{R},X|C}(\alpha_t)}\subseteq \overline{\Delta_{\mathbb{R}}(\alpha)}\bigcap t\times \mathbb{R}.
$$
However, since both $\Delta_{\mathbb{R},X}(\alpha)$ and $\Delta_{\mathbb{R},X|C}(\alpha_t)$ are closed convex sets in $\mathbb{R}^2$ and $\mathbb{R}$, we have
$$
t\times \overline{\Delta_{\mathbb{R},X|C}(\alpha_t)}= \overline{\Delta_{\mathbb{R}}(\alpha)}\bigcap t\times \mathbb{R},
$$
therefore
\begin{equation}
\label{slice}
t\times \Delta_{X|C}(\alpha_t)=\Delta(\alpha)\bigcap t\times \mathbb{R}.
\end{equation}
Let
$$
f(t)=\nu_x(N_t|_C)\ ,\ g(t)=Z_t\cdot C+\nu_x(N_t|_C),
$$
then $\Delta(\alpha)\bigcap [0,s)\times \mathbb{R}$ is the region bounded by the graphs of $f(t)$ and $g(t)$.

Now we prove the piecewise linear property of $f(t)$ and $g(t)$. By Lemma \ref{component}, we have $N_{t_1}\leq N_{t_2}$ if $0\leq t_1\leq t_2<s$, thus $f(t)$ is increasing. Since $N_t$ is an exceptional divisor by Theorem \ref{orthogonal}, the number of the prime components of $N_t$ is uniformly bounded by the Picard number $\rho(X)$. Thus we can denote $N_t=\sum_{i=1}^{r}a_i(t)N_i$, where $a_i(t)\geq 0$ is an increasing and continuous function. Moreover, there exsists $0=t_0<t_1<\ldots<t_k=s$ such that the prime components of $N_t$ are the same when $t$ lies in the interval $(t_i,t_{i+1})$ for $i=0,\ldots,k-1$, and the number of prime components of $N_t$ will increase at every $t_i$ for $i=1,\ldots,k-1$. We write $s_i=\frac{t_{i-1}+t_i}{2}$ for $i=1,\ldots,k$.

We denote the linear subspace of $H^{1,1}(X,\mathbb{R})$ spanned by the prime components of $N_{s_i}$ by $V_i$, and let $V_i^{\bot}$ be the orthogonal space of $V_i$ with respect to $q$. By the proof of Lemma \ref{component},  for $t\in (t_{i-1},t_i)$ we have
\begin{eqnarray}
\label{decompose}
Z_t=Z_{s_i}+(s_i-t)\{C\}_i^\bot\\
\label{decompse 2}
N_t=N_{s_i}+(t_i-t)C_i^{\parallel},
\end{eqnarray}
where $\{C\}_i^\bot$ is the projection of $\{C\}$ to $V_i^\bot$, and $C_i^\parallel$ is a linear combination of the prime components of $N_{s_i}$ satisfying that the cohomology class $\{C_i^\parallel\}$ is equal to the projection of $\{C\}$ to $V_i$. By Theorem \ref{contain exceptional}, the prime components of $N_{s_i}$ are independent, thus $C_i^\parallel$ is uniquely defined.
The piecewise linearity property of $f(t)$ and $g(t)$ follows directly from (\ref{decompose}) and (\ref{decompse 2}), and thus $f(t)$ and $g(t)$ can be continuously extended to $s$. Therefore we conclude that $\Delta(\alpha)$ is the region bounded by the graphs of $f(t)$ and $g(t)$ for $t\in[0,s]$. Thus the vertices of $\Delta(\alpha)$ are contained in the set  $\{(t_i,f(t_i)),(t_j,g(t_j))\in \mathbb{R}^2\mid i,j=0,\ldots,k\}$.  This
means that a vertex of $\Delta(\alpha)$ may only occur for those $t\in[0,s]$, where a new curve appears in $N_t$. Since $r\leq \rho(X)$, the number of vertices is bounded by $2\rho(X)+2$.
The fact that $f(t)$ is convex and $g(t)$ concave is a consequence of the convexity of $\Delta(\alpha)$.

By (\ref{slice}), we have
\[\begin{split}
2\vol_{\mathbb{R}^2}(\Delta(\alpha))&=2\int_{0}^{s}\vol_\mathbb{R}(\Delta_{X|C}(\alpha_t))\mathrm{d}t\\
&=2\int_{0}^{s}Z_t\cdot C\mathrm{d}t\\
&=\vol_X(\alpha)-\vol_X(\alpha-s C)\\
&=\vol_X(\alpha).
\end{split} \]
where the second equality follows by Proposition \ref{restrict volume}, the third one by Theorem \ref{differential} and the last one by Remark \ref{boundary}. We have proved the theorem under the assumption that $C$ is nef.

(ii) Now we prove the theorem when $C$ is not nef, i.e., $C^2<0$. Recall that $a:=\sup\{t>0\mid C\subseteq E_{nK}(\alpha)\}$. By (\ref{same nK}), if  $C\subseteq E_{nK}(\alpha_t)$ for some $t\in[0,s)$, we have $C\subseteq E_{nK}(Z(\alpha_t))$. By the proof in Theorem \ref{differential 2} we have
\begin{gather}
Z(\alpha_s)\cdot C=0,\nonumber\\
Z(\alpha_s)=Z(\alpha_t),\nonumber
\end{gather}
for $0\leq s\leq t$. Thus we have
$$
\{0\leq t <s\mid C\not\subseteq  E_{nK}(\alpha_t)\}=(a,s),
$$
and $\Delta(\alpha)$ is contained in $[a,s]\times \mathbb{R}$. By Theorem \ref{differential 2} we also have
\[\begin{split}
2\vol_{\mathbb{R}^2}(\Delta(\alpha))&=2\int_{a}^{s}\vol_\mathbb{R}(\Delta_{X|C}(\alpha_t))\mathrm{d}t\\
&=2\int_{a}^{s}Z_t\cdot C\mathrm{d}t\\
&=\vol_X(\alpha_a)-\vol_X(\alpha_s)\\
&=\vol_X(\alpha).
\end{split} \]
Since the prime components of $N_{t_1}$ is contained in that of $N_{t_2}$ if $a<t_1\leq t_2<s$, using the same arguments above, we obtain the piecewise linear property of $f(t)$ and $g(t)$ which can also be extended to $s$. The theorem is proved completely. \end{proof}

\begin{rem}
If $X$ is a projective surface, by the main result in \cite{BKS03}, the cone of big divisors of $X$ admits a locally finite decomposition into locally polyhedral subcones such that the support of the negative part in the Zariski decomposition is constant on each subcone. It is noticeable that if we only assume $X$ to be K\"ahler, this decomposition still holds if we replace the cone of big divisors by the cone of big classes and use divisorial Zariski decomposition instead. This property ensures that the generalized Okounkov bodies should also be polygons.
\end{rem}

\subsection{Generalized Okounkov bodies for pseudo-effective classes}

Throughout this subsection, $X$ will stand for a K\"ahler surface if not specially mentioned. Our main goal in this subsection is to study the behavior of generalized Okounkov bodies on the boundary of the big cone. 
\begin{defin}
Let $X$ be any K\"ahler manifold, if $\alpha\in H^{1,1}(X,\mathbb{R})$ is any pseudo-effective class. We define the \emph{generalized Okounkov body} $\Delta(\alpha)$ with respect to the fixed flag by
$$
\Delta(\alpha):=\bigcap_{\epsilon>0}\Delta(\alpha+\epsilon\omega),
$$
where $\omega$ is any K\"ahler class.
\end{defin} 
It is easy to check that our definition does not depend on the choice of $\omega$, and if $\alpha$ is big, by Proposition \ref{body continuity}, the definition is consistent with Definition \ref{generalized Okounkov}. Now we recall the definition of numerical dimension for any real (1,1)-class.
\begin{defin}[numerical dimension]
Let $X$ be a compact K\"ahler manifold. For a class $\alpha\in H^{1,1}(X,\mathbb{R})$, the \emph{numerical dimension} $n(\alpha)$ is defined to be $-\infty$ if $\alpha$ is not pseudo-effective, and
$$
n(\alpha)=\text{\rm max}\{p\in\mathbb{N},\langle \alpha^p\rangle\neq 0\},
$$
if $\alpha$ is pseudo-effective. 
\end{defin} 
We recall that the right-hand side of the equation above involves the \textit{positive intersection product} $\langle \alpha^{p}\rangle\in H^{p,p}_{\geq0}(X,\mathbb{R})$ defined in \cite{BDPP13}. When $X$ is a K\"ahler surface, we simply have 
$$
n(\alpha)=\text{\rm max}\{p\in\mathbb{N}, Z(\alpha)^p\neq 0\},\ \  p \in \{0,1,2\}.
$$
If $n(\alpha)=2$, $\alpha$ is big and the situation is studied in the last subsection. Throughout this subsection, we assume $\alpha\in\mathcal{\partial E}$.

\begin{lem}\it
\label{construct nef}
Let $\{N_1,\ldots,N_r\}$ be an exceptional family of prime divisors, $\omega$ be any K\"ahler class. Then there exists unique positive numbers $b_1,\ldots,b_r$ such that $\omega+\sum_{i=1}^{r}b_iN_i$ is big and nef satisfying $\text{\rm Null}(\omega+\sum_{i=1}^{r}b_iN_i)=\bigcup_{i=1}^{r}N_i$.
\end{lem}

\begin{proof}
If we set
\begin{equation}
\left(
  \begin{array}{c}
    b_1\\
    \vdots\\
    b_r\\
  \end{array}
\right)=
-S^{-1}\cdot\left(
  \begin{array}{c}
    \omega\cdot N_1\\
    \vdots\\
    \omega\cdot N_r\\
  \end{array}
\right)\nonumber,
\end{equation}
where $S$ denotes the intersection matrix of $\{N_1,\ldots, N_r\}$, we have $(\omega+\sum_{i=1}^{r}b_iN_i)\cdot N_j=0$ for $j=1,\ldots,r$. By Lemma \ref{BKS}, we conclude that all $b_i$ are positive and thus $\omega+\sum_{i=1}^{r}b_iN_i$ is big and nef. 
\end{proof}

\begin{proposition}\it
\label{Zariski pseudo}
Let $\alpha$ be any pseudo-effective class with $N(\alpha)=\sum_{i=1}^{r}a_iN_i$, $\omega$ be a K\"ahler class. Then for $\epsilon>0$ small enough, we have the divisorial Zariski decomposition 
\begin{eqnarray}
Z(\alpha+\epsilon\omega)=Z(\alpha)+\epsilon(\omega+\sum_{i=1}^{r}b_iN_i),\nonumber\\
N(\alpha+\epsilon\omega)=\sum_{i=1}^{r}(a_i-\epsilon b_i)N_i,\nonumber
\end{eqnarray}
where $b_i$ is the positive number defined in Lemma \ref{construct nef}.
\end{proposition}

\begin{proof}
Since $Z(\alpha)+\epsilon(\omega+\sum_{i=1}^{r}b_iN_i)$ is nef and orthogonal to all $N_i$ by Lemma \ref{construct nef}, by Theorem \ref{orthogonal}, if $\epsilon$ satisfies that $a_i-\epsilon b_i>0$ for all $i$, the divisorial decomposition in the proposition holds.
\end{proof}

If $n(\alpha)=0$, we have $Z(\alpha)=0$ and thus $\alpha=\sum_{i=1}^{r} a_iN_i$ is an exceptional effective $\mathbb{R}$-divisor. We fix a flag
$$
X\supseteq C\supseteq \{x\},
$$
where $C\neq N_i$ for all $i$. Then we have
\begin{thm}\it
For any pseudo-effective class $\alpha$ whose numerical dimension $n(\alpha)=0$, we have 
$$\Delta_{(C,x)}(\alpha)=0\times \nu_x(N(\alpha)|_C).$$
\end{thm} 
\begin{proof}
We asumme $N(\alpha)=\sum_{i=1}^{r}a_iN_i$. Fix a K\"ahler class $\omega$, by Proposition \ref{Zariski pseudo}, we have
\begin{eqnarray}
Z(\alpha+\epsilon\omega)=\epsilon(\omega+\sum_{i=1}^{r}b_iN_i)\label{0z},\\
N(\alpha+\epsilon\omega)=\sum_{i=1}^{r}(a_i-\epsilon b_i)N_i\label{0n},
\end{eqnarray}
where $b_i$ is the positive number defined in Lemma \ref{construct nef}. Since $T\mapsto T-N(\alpha+\epsilon\omega)$ is a bijection between the positive currents in $\alpha+\epsilon\omega$ and those in $Z(\alpha+\epsilon\omega)$, we have
$$
\Delta(\alpha+\epsilon\omega)=\epsilon\Delta(\omega+\sum_{i=1}^{r}b_iN_i)+\nu(\sum_{i=1}^{r}(a_i-\epsilon b_i)N_i),
$$
where $\nu(\sum_{i=1}^{r}(a_i-\epsilon b_i)N_i)=\nu_{(C,x)}(\sum_{i=1}^{r}(a_i-\epsilon b_i)N_i)$ is the valuation-like function defined in  Section \ref{defin}. Thus the diameter of $\Delta(\alpha+\epsilon\omega)$ converges to 0 when $\epsilon$ tends to 0, and we conclude that $\Delta(\alpha)$ is a single point in $\mathbb{R}^2$. Since
\begin{eqnarray}
\Delta(\alpha+\epsilon\omega)\bigcap 0\times \mathbb{R}&=& 0\times\Delta_{X|C}(\alpha+\epsilon\omega)\nonumber\\
&=&0\times [\nu_x(N(\alpha+\epsilon\omega)|_C),\nu_x(N(\alpha+\epsilon\omega)|_C)+Z(\alpha+\epsilon\omega)\cdot C]\label{inter}\nonumber,
\end{eqnarray}
by (\ref{0z}) and (\ref{0n}) we have 

\begin{eqnarray}
\Delta(\alpha)\bigcap0\times \mathbb{R}=0\times \nu_x(\sum_{i=1}^{r}a_iN_i|_C)\nonumber,
\end{eqnarray}
and we prove the first part of Theorem \ref{okounkov psf}..
\end{proof}

If $n(\alpha)=1$, $Z(\alpha)$ is nef but not big. If there exists one irreducible curve $C$ such that $Z(\alpha)\cdot C>0$, we fix the flag
$$
X\supseteq C\supseteq \{x\},
$$
then we have
\begin{thm}\it
For any pseudo-effective class $\alpha$ whose numerical dimension $n(\alpha)=1$, we have 
$$\Delta(\alpha)=0\times [\nu_x(N(\alpha)|_C),\nu_x(N(\alpha)|_C)+Z(\alpha)\cdot C]\nonumber.$$
\end{thm}
\begin{proof}
By the assumption $Z(\alpha)\cdot C>0$ we know that $C\not\subseteq {\rm Supp}(N(\alpha))$. By Proposition \ref{Zariski pseudo}, when $\epsilon$ small enough, the divisorial Zariski decomposition for $\alpha+\epsilon\omega$ is
\begin{eqnarray}
Z(\alpha+\epsilon\omega)=Z(\alpha)+\epsilon(\omega+\sum_{i=1}^{r}b_iN_i),\label{z}\\
N(\alpha+\epsilon\omega)=\sum_{i=1}^{r}(a_i-\epsilon b_i)N_i,\label{n}
\end{eqnarray}
where $b_i$ is the positive number defined in Lemma \ref{construct nef}. Combine (\ref{z}) and (\ref{n}), we have
\begin{eqnarray}
\Delta(\alpha)\bigcap0\times \mathbb{R}&=& \bigcap_{\epsilon>0}\Delta(\alpha+\epsilon\omega)\bigcap 0\times \mathbb{R}\nonumber\\
&=& \bigcap_{\epsilon>0}0\times [\nu_x(N(\alpha+\epsilon\omega)|_C),\nu_x(N(\alpha+\epsilon\omega)|_C)+Z(\alpha+\epsilon\omega)\cdot C]\nonumber\\
&=& 0\times [\nu_x(\sum_{i=1}^{r}a_iN_i|_C),\nu_x(\sum_{i=1}^{r}a_iN_i|_C)+Z(\alpha)\cdot C]\nonumber.
\end{eqnarray}
Since we have 
$$\vol_{\mathbb{R}^2}(\Delta(\alpha))=\lim\limits_{\epsilon\rightarrow 0}\vol_{\mathbb{R}^2}(\Delta(\alpha+\epsilon\omega))=\lim\limits_{\epsilon\rightarrow 0}Z(\alpha+\epsilon\omega)^2=0,$$
and $\Delta(\alpha)$ is a closed convex set, we conclude that there are no points of $\Delta(\alpha)$ which lie outside $0\times \mathbb{R}$ as $\vol_{\mathbb{R}}(\Delta(\alpha)\bigcap0\times\mathbb{R})=Z(\alpha)\cdot C>0$. We finish the proof of Theorem \ref{okounkov psf}.
\end{proof}

\footnotesize
\noindent\textit{Acknowledgments.}
I would like to express my warmest gratitude to my thesis supervisor Professor Jean-Pierre Demailly for his many valuable suggestions and help in this work. I also would like to thank Professor Sen Hu for his constant encouragement. This research is supported by the China Scholarship Council.

\end{document}